\documentclass{article}%
\usepackage{amssymb}
\usepackage{graphicx}
\usepackage{amsmath}%
\usepackage{amsfonts}
\providecommand{\U}[1]{\protect\rule{.1in}{.1in}}
\providecommand{\U}[1]{\protect\rule{.1in}{.1in}}
\newtheorem{theorem}{Theorem}

\newtheorem{corollary}[theorem]{Corollary}

\newtheorem{definition}[theorem]{Definition}

\newtheorem{lemma}[theorem]{Lemma}
\newtheorem{notation}[theorem]{Notation}

\newtheorem{proposition}[theorem]{Proposition}
\newtheorem{remark}[theorem]{Remark}

\newenvironment{proof}[1][Proof]{\textbf{#1.} }{\ \rule{0.5em}{0.5em}}
\begin{document}

\title{Differential Geometry of Microlinear Fr\"{o}licher Spaces III}
\author{Hirokazu Nishimura\\Institute of Mathematics, University of Tsukuba\\Tsukuba, Ibaraki, 305-8571\\Japan}
\maketitle

\begin{abstract}
As the third of our series of papers on differential geometry of microlinear
Fr\"{o}licher spaces is this paper devoted to the Fr\"{o}licher-Nijenhuis
calculus of their named bracket. The main result is that the
Fr\"{o}licher-Nijenhuis bracket satisfies the graded Jacobi identity. It is
also shown that the Lie derivation preserves the Fr\"{o}licher-Nijenhuis
bracket. Our definitions and discussions are highly geometric, \ while
Fr\"{o}licher and Nijenhuis' original definitions and discussions were highly algebraic.

\end{abstract}

\section{Introduction}

As Mangiarotti and Modugno \cite{mm}\ have amply demonstrated, the central
part of orthodox differential geometry based on principal connections can be
developed within a more general framework of fibered manifolds (without any
distinguished additional structures), in which the graded Lie algebra of
tangent-vector-valued forms investigated by Fr\"{o}licher and Nijenhuis
\cite{fn}\ renders an appropriate differential calculus. The present paper is
concerned with this graded Lie algebra, which plays a crucial role in their
general differential geometry. Our present approach as well as \cite{nishi-bc}%
\ is highly combinatorial or geometric, while Fr\"{o}licher and Nijenhuis'
original approach was tremendously algebraic.

This paper consists of 4 sections, besides this introduction. The second
section is devoted to preliminaries including vector fields and (real-valued)
differential forms. Obviously tangent-vector-valued forms are a generalization
of differential forms and vector fields at the same time, while the graded Lie
algebra of tangent-vector-valued forms is a generaliztion of the Lie algebra
of vector fields. Section 3 gives two distinct but equivalent views of
tangent-vector-valued forms, just as we have given two distinct but equivalent
views of vector fields in \cite{nishi-e}. Section 4 is divided into two
subsections, the first of which is mainly concerned with the graded Jacobi
identity of entities much more general than tangent-vector-valued forms (i.e.,
without homogeneity or the alternating property assumed at all), while the
second of which derives the graded Jacobi identity of tangent-vector-valued
forms from the highly general graded Jacobi identity established in the first
subsection. Our proof of the graded Jacobi identity in the first subsection is
based upon the general Jacobi identity established by the author
\cite{nishi-a}\ more than a decade ago. Section 5 is similarly divided into
two subsections, the first of which shows that the Lie derivation of
differential semiforms by tangent-vector-valued semiforms preserves the Lie
bracket, while the second of which demonstrates that the Lie derivation of
differential forms by tangent-vector-valued forms preserves the
Fr\"{o}licher-Nijenhuis bracket.

\section{Preliminaries}

In this paper, $n,p,q,...$ represent natural numbers. We assume that the
reader has already read our previous papers \cite{nishi-e} and \cite{nishi-f}.

\subsection{Fr\"{o}licher Spaces}

Fr\"{o}licher and his followers have vigorously and consistently developed a
general theory of smooth spaces, often called \textit{Fr\"{o}licher spaces}
for his celebrity, which were intended to be the \textit{underlying set
theory} for infinite-dimensional differential geometry. A Fr\"{o}licher space
is an underlying set endowed with a class of real-valued functions on it
(simply called \textit{structure} \textit{functions}) and a class of mappings
from the set $\mathbb{R}$ of real numbers to the underlying set (called
\textit{structure} \textit{curves}) subject to the condition that structure
curves and structure functions should compose so as to yield smooth mappings
from $\mathbb{R}$ to itself. It is required that the class of structure
functions and that of structure curves should determine each other so that
each of the two classes is maximal with respect to the other as far as they
abide by the above condition. What is most important among many nice
properties about the category $\mathbf{FS}$ of Fr\"{o}licher spaces and smooth
mappings is that it is cartesian closed, while neither the category of
finite-dimensional smooth manifolds nor that of infinite-dimensional smooth
manifolds modelled after any infinite-dimensional vector spaces such as
Hilbert spaces, Banach spaces, Fr\'{e}chet spaces or the like is so at all.
For a standard reference on Fr\"{o}licher spaces the reader is referred to
\cite{fro}.

\subsection{Weil Algebras and Infinitesimal Objects}

The notion of a \textit{Weil algebra} was introduced by Weil himself in
\cite{wei}. We denote by $\mathbf{W}$ the category of Weil algebras. Roughly
speaking, each Weil algebra corresponds to an infinitesimal object in the
shade. By way of example, the Weil algebra $\mathbb{R}[X]/(X^{2})$ (=the
quotient ring of the polynomial ring $\mathbb{R}[X]$\ of an indeterminate
$X$\ over $\mathbb{R}$ modulo the ideal $(X^{2})$\ generated by $X^{2}$)
corresponds to the infinitesimal object of first-order nilpotent
infinitesimals, while the Weil algebra $\mathbb{R}[X]/(X^{3})$ corresponds to
the infinitesimal object of second-order nilpotent infinitesimals. Although an
infinitesimal object is undoubtedly imaginary in the real world, as has
harassed both mathematicians and philosophers of the 17th and the 18th
centuries (because mathematicians at that time preferred to talk infinitesimal
objects as if they were real entities), each Weil algebra yields its
corresponding \textit{Weil functor} on the category of smooth manifolds of
some kind to itself, which is no doubt a real entity. By way of example, the
Weil algebra $\mathbb{R}[X]/(X^{2})$ yields the tangent bundle functor as its
corresponding Weil functor. Intuitively speaking, the Weil functor
corresponding to a Weil algebra stands for the exponentiation by the
infinitesimal object corresponding to the Weil algebra at issue. For Weil
functors on the category of finite-dimensional smooth manifolds, the reader is
referred to \S 35 of \cite{kolar}, while the reader can find a readable
treatment of Weil functors on the category of smooth manifolds modelled on
convenient vector spaces in \S 31 of \cite{kri}.

\textit{Synthetic differential geometry }(usually abbreviated to SDG), which
is a kind of differential geometry with a cornucopia of nilpotent
infinitesimals, was forced to invent its models, in which nilpotent
infinitesimals were visible. For a standard textbook on SDG, the reader is
referred to \cite{lav}, while he or she is referred to \cite{kock} for the
model theory of SDG constructed vigorously by Dubuc \cite{dub} and others.
Although we do not get involved in SDG herein, we will exploit locutions in
terms of infinitesimal objects so as to make the paper highly readable. Thus
we prefer to write $\mathcal{W}_{D}$\ and $\mathcal{W}_{D_{2}}$\ in place of
$\mathbb{R}[X]/(X^{2})$ and $\mathbb{R}[X]/(X^{3})$ respectively, where $D$
stands for the infinitesimal object of first-order nilpotent infinitesimals,
and $D_{2}$\ stands for the infinitesimal object of second-order nilpotent
infinitesimals. To Newton and Leibniz, $D$ stood for
\[
\{d\in\mathbb{R}\mid d^{2}=0\}
\]
while $D_{2}$\ stood for
\[
\{d\in\mathbb{R}\mid d^{3}=0\}
\]
We will write $\mathcal{W}_{d\in D_{2}\mapsto d^{2}\in D}$ for the homomorphim
of Weil algebras $\mathbb{R}[X]/(X^{2})\rightarrow\mathbb{R}[X]/(X^{3})$
induced by the homomorphism $X\rightarrow X^{2}$ of the polynomial ring
\ $\mathbb{R}[X]$ to itself. Such locutions are justifiable, because the
category $\mathbf{W}$ of Weil algebras in the real world and the category of
infinitesimal objects in the shade are dual to each other in a sense. Thus we
have a contravariant functor $\mathcal{W}$\ from the category of infinitesimal
objects in the shade to the category of Weil algebras in the real world. Its
inverse contravariant functor from the category of Weil algebras in the real
world to the category of Weil algebras in the real world is denoted by
$\mathcal{D}$. By way of example, $\mathcal{D}_{\mathbb{R}[X]/(X^{2})}$ and
$\mathcal{D}_{\mathbb{R}[X]/(X^{3})}$\ stand for $D$ and $D_{2}$%
\ respectively. To familiarize himself or herself with such locutions, the
reader is strongly encouraged to read the first two chapters of \cite{lav},
even if he or she is not interested in SDG at all.

In \cite{nishi-c} we have discussed how to assign, to each pair $(X,W)$\ of a
Fr\"{o}licher space $X$ and a Weil algebra $W$,\ another Fr\"{o}licher space
$X\otimes W$\ called the \textit{Weil prolongation of} $X$ \textit{with
respect to} $W$, which is naturally extended to a bifunctor $\mathbf{FS}%
\times\mathbf{W\rightarrow FS}$, and then to show that the functor
$\cdot\otimes W:\mathbf{FS\rightarrow FS}$ is product-preserving for any Weil
algebra $W$. Weil prolongations are well-known as \textit{Weil functors} for
finite-dimensional and infinite-dimensional smooth manifolds in orthodox
differential geometry, as we have already discussed above.

The central object of study in SDG is \textit{microlinear} spaces. Although
the notion of a manifold (=a pasting of copies of a certain linear space) is
defined on the local level, the notion of microlinearity is defined on the
genuinely infinitesimal level. For the historical account of microlinearity,
the reader is referred to \S \S 2.4 of \cite{lav} or Appendix D of
\cite{kock}. To get an adequately restricted cartesian closed subcategory of
Fr\"{o}licher spaces, we have emancipated microlinearity from within a
well-adapted model of SDG to Fr\"{o}licher spaces in the real world in
\cite{nishi-d}. Recall that a Fr\"{o}licher space $X$ is called
\textit{microlinear} providing that any finite limit diagram $\mathbb{D}$ in
$\mathbf{W}$ yields a limit diagram $X\otimes\mathbb{D}$ in $\mathbf{FS}$,
where $X\otimes\mathbb{D}$ is obtained from $\mathbb{D}$ by putting $X\otimes$
to the left of every object and every morphism in $\mathbb{D}$.

As we have discussed there, all convenient vector spaces are microlinear, so
that all $C^{\infty}$-manifolds in the sense of \cite{kri} (cf. Section 27)
are also microlinear.

We have no reason to hold that all Fr\"{o}licher spaces credit Weil
prolongations as exponentiation by infinitesimal objects in the shade.
Therefore we need a notion which distinguishes Fr\"{o}licher spaces that do so
from those that do not. A Fr\"{o}licher space $X$ is called \textit{Weil
exponentiable }if
\begin{equation}
(X\otimes(W_{1}\otimes_{\infty}W_{2}))^{Y}=(X\otimes W_{1})^{Y}\otimes
W_{2}\label{2.2.1}%
\end{equation}
holds naturally for any Fr\"{o}licher space $Y$ and any Weil algebras $W_{1}$
and $W_{2}$. If $Y=1$, then (\ref{2.2.1}) degenerates into
\begin{equation}
X\otimes(W_{1}\otimes_{\infty}W_{2})=(X\otimes W_{1})\otimes W_{2}%
\label{2.2.2}%
\end{equation}
If $W_{1}=\mathbb{R}$, then (\ref{2.2.1}) degenerates into
\begin{equation}
(X\otimes W_{2})^{Y}=X^{Y}\otimes W_{2}\label{2.2.3}%
\end{equation}

We have shown in \cite{nishi-c} that all convenient vector spaces are Weil
exponentiable, so that all $C^{\infty}$-manifolds in the sense of \cite{kri}
(cf. Section 27) are Weil exponentiable.

We have demonstrated in \cite{nishi-d} that all Fr\"{o}licher spaces that are
microlinear and Weil exponentiable form a cartesian closed category. In the
sequel, $M$ shall be assumed to be such a Fr\"{o}licher space.

It is well known that the category $\mathbf{W}$\ is left exact. In SDG, a
finite diagram $\mathbb{D}$ in $\mathbf{D}$\textbf{\ }is called a
\textit{quasi-colimit diagram} provided that the contravariant functor
$\mathcal{W}$ transforms $\mathbb{D}$\ into a limit diagram in $\mathbf{W}$.
By way of example, the following diagram in $\mathbf{D}$\textbf{\ }is a famous
quasi-colimit diagram, for which the reader is referred to pp.92-93 of
\cite{lav}.%

\begin{equation}%
\begin{array}
[c]{ccc}%
D(2) &
\begin{array}
[c]{c}%
i\\
\rightarrow\\
\,
\end{array}
& D^{2}\\%
\begin{array}
[c]{ccc}%
i & \downarrow & \,
\end{array}
&  &
\begin{array}
[c]{ccc}%
\, & \downarrow & \psi
\end{array}
\\
D^{2} &
\begin{array}
[c]{c}%
\,\\
\rightarrow\\
\varphi
\end{array}
& D^{2}\oplus D
\end{array}
\label{2.2.4}%
\end{equation}
where $i:D(2)\rightarrow D^{2}$ is the canonical injection, $D^{2}\oplus D$
is
\[
D^{2}\oplus D=\{(d_{1},d_{2},e)\in D^{3}\mid d_{1}e=d_{2}e=0\}
\]
$\varphi:D^{2}\rightarrow D^{2}\oplus D$ is
\[
\varphi(d_{1},d_{2})=(d_{1},d_{2},0)
\]
for any $(d_{1},d_{2})\in D^{2}$, and $\psi:D^{2}\rightarrow D^{2}\oplus D$
is
\[
\psi(d_{1},d_{2})=(d_{1},d_{2},d_{1}d_{2})
\]
for any $(d_{1},d_{2})\in D^{2}$.

\subsection{Vector Fields}

Our two distinct but equivalent viewpoints of vector fields on $M$ are simply
based upon the following exponential law:
\begin{align*}
&  \left[  M\rightarrow M\otimes\mathcal{W}_{D}\right] \\
&  =\left[  M\rightarrow M\right]  \otimes\mathcal{W}_{D}%
\end{align*}
The first definition of a vector field on $M$ goes as follows:

\begin{definition}
A vector field on $M$ is a section of the tangent bundle $\pi:M\otimes
\mathcal{W}_{D}\rightarrow M$.
\end{definition}

The second definition of a vector field on $M$ goes as follows:

\begin{definition}
A vector field on $M$ is a tangent vector of the space $\left[  M\rightarrow
M\right]  $ with foot point $\mathrm{id}_{M}$.
\end{definition}

Generally speaking, we prefer the second definition of a vector field to the
first one. In our previous paper \cite{nishi-e}, we have shown that

\begin{theorem}
\label{t2.3.1}The totality of vector fields on $M$ forms a Lie algebra.
\end{theorem}

In particular, our proof of the Jacobi identity of vector fields is based upon
the following general Jacobi identity.

\begin{theorem}
\label{t2.3.2}Let $\gamma_{123},\gamma_{132},\gamma_{213},\gamma_{231}%
,\gamma_{312},\gamma_{321}\in M\otimes W_{D^{3}}$. As long as the following
three expressions are well defined, they sum up only to vanish:
\begin{align*}
&  (\gamma_{123}\overset{\cdot}{\underset{1}{-}}\gamma_{132})\overset{\cdot
}{-}(\gamma_{231}\overset{\cdot}{\underset{1}{-}}\gamma_{321})\\
&  (\gamma_{231}\overset{\cdot}{\underset{2}{-}}\gamma_{213})\overset{\cdot
}{-}(\gamma_{312}\overset{\cdot}{\underset{2}{-}}\gamma_{132})\\
&  (\gamma_{312}\overset{\cdot}{\underset{3}{-}}\gamma_{321})\overset{\cdot
}{-}(\gamma_{123}\overset{\cdot}{\underset{3}{-}}\gamma_{213})
\end{align*}

\end{theorem}

The above theorem was discovered by the author in \cite{nishi-a}\ more than a
decade ago, and with due regard to its importance, it was provided with two
other proofs in \cite{nishi-b} and \cite{nishi-oso}.

\subsection{Euclidean Vector Spaces}

Frankly speaking, our exposition of a Euclidean vector space in \cite{nishi-f}
was a bit confused. The exact definition of a Euclidean vector space goes as follows.

\begin{definition}
A vector space $\mathbb{E}$\ (over $\mathbb{R}$) in the category $\mathbf{FS}
$\ is called Euclidean provided that the canonical mapping $\mathbf{i}%
_{\mathbb{E}}^{1}:\mathbb{E\times E\rightarrow E}\otimes\mathcal{W}_{D}$
induced by the mapping
\[
(\mathbf{a},\mathbf{b})\in\mathbb{E\times E\mapsto(}x\in\mathbb{R}%
\mapsto\mathbf{a}\mathbb{+}x\mathbf{b}\in\mathbb{E)\in E}^{\mathbb{R}}%
\]
is bijective.
\end{definition}

\begin{notation}
Let $\mathbb{E}$\ be a Euclidean vector space. Given $\gamma\in\mathbb{E}%
\otimes\mathcal{W}_{D}$, we write $\mathbf{D}\left(  \gamma\right)  $ for
$\mathbf{b}\in\mathbb{E}$\ in the above definition.
\end{notation}

\begin{notation}
Let $\mathbb{E}$\ be a Euclidean vector space. Given $\gamma\in\mathbb{E}%
\otimes\mathcal{W}_{D^{n}}$ ($n\geq2$), we write $\mathbf{D}_{i}\left(
\gamma\right)  \in\mathbb{E}\otimes\mathcal{W}_{D^{n-1}}$ for the image of
$\gamma$\ under the composite of mappings
\begin{align*}
&  \mathbb{E}\otimes\mathcal{W}_{D^{n}}\underrightarrow{\quad\mathrm{id}%
_{\mathbb{E}}\otimes\mathcal{W}_{(d_{1},...d_{n})\in D^{n}\mapsto
(d_{1},...,d_{i-1},d_{i+1},...,d_{n},d_{i})\in D^{n}}\quad}\mathbb{E}%
\otimes\mathcal{W}_{D^{n}}\\
&  =\mathbb{E}\otimes\left(  \mathcal{W}_{D^{n-1}}\otimes_{\infty}%
\mathcal{W}_{D}\right)  =\left(  \mathbb{E}\otimes\mathcal{W}_{D^{n-1}%
}\right)  \otimes\mathcal{W}_{D}\underrightarrow{\quad\mathbf{D\quad}%
}\mathbb{E}\otimes\mathcal{W}_{D^{n-1}}%
\end{align*}

\end{notation}

\begin{theorem}
\label{t2.4.1}In a Euclidean vector space $\mathbb{E}$, Taylor's expansion
theorem holds in the sense that the canonical mapping $\mathbf{i}_{\mathbb{E}%
}^{2}:\mathbb{E\times E\times E\times E\rightarrow E}\otimes\mathcal{W}%
_{D^{2}}$ induced by the mapping
\begin{align*}
(\mathbf{a,b}_{1},\mathbf{b}_{2},\mathbf{b}_{12})  &  \in\mathbb{E\times
E\times E\times E\mapsto}\\
\left(  \left(  x_{1},x_{2}\right)  \in\mathbb{R}^{2}\mapsto\mathbf{a+}%
x\mathbf{_{1}b}_{1}+x_{2}\mathbf{b}_{2}+x_{1}x_{2}\mathbf{b}_{12}\in
\mathbb{E}\right)   &  \in\mathbb{E}^{\mathbb{R}^{2}}%
\end{align*}
is bijective, the canonical mapping $\mathbf{i}_{\mathbb{E}}^{3}%
:\mathbb{E\times E\times E\times E\times E\times E\times E\times E\rightarrow
E}\otimes\mathcal{W}_{D^{3}}$ induced by the mapping
\begin{align*}
(\mathbf{a,b}_{1},\mathbf{b}_{2},\mathbf{b}_{3},\mathbf{b}_{12},\mathbf{b}%
_{13},\mathbf{b}_{23},\mathbf{b}_{123})  &  \in\mathbb{E\times E\times E\times
E\times E\times E\times E\times E\mapsto}\\
(\left(  x_{1},x_{2},x_{3}\right)   &  \in\mathbb{R}^{3}\mapsto\mathbf{a+}%
x\mathbf{_{1}b}_{1}+x_{2}\mathbf{b}_{2}+x_{3}\mathbf{b}_{3}\\
+x_{1}x_{2}\mathbf{b}_{12}+x_{1}x_{3}\mathbf{b}_{13}+x_{2}x_{3}\mathbf{b}%
_{23}+x_{1}x_{2}x_{3}\mathbf{b}_{123}  &  \in\mathbb{E)}\in\mathbb{E}%
^{\mathbb{R}^{3}}%
\end{align*}
is bijective, and so on.
\end{theorem}

\begin{proof}
Here we deal only with the first case, leaving similar treatments of the other
cases to the reader. Schematically we have
\begin{align*}
&  \mathbb{E}\otimes\mathcal{W}_{D^{2}}\\
&  =\mathbb{E}\otimes\left(  \mathcal{W}_{D}\otimes_{\infty}\mathcal{W}%
_{D}\right) \\
&  =\left(  \mathbb{E}\otimes\mathcal{W}_{D}\right)  \otimes\mathcal{W}_{D}\\
&  =\left(  \mathbb{E\times E}\right)  \otimes\mathcal{W}_{D}\\
&  =\left(  \mathbb{E}\otimes\mathcal{W}_{D}\right)  \mathbb{\times}\left(
\mathbb{E}\otimes\mathcal{W}_{D}\right) \\
&  \text{[Since the endofunctor }\cdot\otimes\mathcal{W}_{D}\text{ of the
category }\mathbf{FS}\text{\ preserves products]}\\
&  =\mathbb{E\times E\times E\times E}%
\end{align*}

\end{proof}

\begin{proposition}
\label{t2.4.2} Let $\mathbb{E}$\ be a Euclidean vector space. Given $\gamma
\in\mathbb{E}\otimes\mathcal{W}_{D^{2}}$, we have
\[
\mathbf{D}\left(  \mathbf{D}_{2}\left(  \gamma\right)  \right)  =\mathbf{D}%
\left(  \mathbf{D}_{1}\left(  \gamma\right)  \right)
\]

\end{proposition}

\begin{proof}
It is easy to see that both sides give rise to $\mathbf{b}_{12}$\ in Theorem
\ref{t2.4.1}.
\end{proof}

Here we will give a slight variant of Taylor's expansion theorem.

\begin{theorem}
\label{t2.4.1'}Let $\mathbb{E}$\ be a Euclidean vector space, which is
microlineaar. The canonical mapping $\mathbf{i}_{\mathbb{E}}^{D^{2}\oplus
D}:\mathbb{E\times E\times E\times E\times E\rightarrow E}\otimes
\mathcal{W}_{D^{2}\oplus D}$ induced by the mapping
\begin{align*}
(\mathbf{a,b}_{1},\mathbf{b}_{2},\mathbf{b}_{12},\mathbf{c})  &
\in\mathbb{E\times E\times E\times E\times E\mapsto}\\
\left(  \left(  x_{1},x_{2},x_{3}\right)  \in\mathbb{R}^{3}\mapsto
\mathbf{a+}x\mathbf{_{1}b}_{1}+x_{2}\mathbf{b}_{2}+x_{1}x_{2}\mathbf{b}%
_{12}+x_{3}\mathbf{c}\in\mathbb{E}\right)   &  \in\mathbb{E}^{\mathbb{R}^{3}}%
\end{align*}
is bijective.
\end{theorem}

\begin{proof}
This follows from Theorem \ref{t2.4.1} and the quasi-colimit diagram
(\ref{2.2.4}).
\end{proof}

\begin{proposition}
\label{t2.4.3}Let $\mathbb{E}$\ be a Euclidean vector space, which is
microlineaar. Let $\gamma_{1},\gamma_{2}\in\mathbb{E}\otimes\mathcal{W}%
_{D^{2}}$ with
\begin{align*}
&  \left(  \mathrm{id}_{\mathbb{E}}\otimes\mathcal{W}_{\left(  d_{1}%
,d_{2}\right)  \in D(2)\mapsto\left(  d_{1},d_{2}\right)  \in D^{2}}\right)
\left(  \gamma_{1}\right) \\
&  =\left(  \mathrm{id}_{\mathbb{E}}\otimes\mathcal{W}_{\left(  d_{1}%
,d_{2}\right)  \in D(2)\mapsto\left(  d_{1},d_{2}\right)  \in D^{2}}\right)
\left(  \gamma_{2}\right)
\end{align*}
Then we have
\[
\mathbf{D}\left(  \gamma_{1}\overset{\cdot}{-}\gamma_{2}\right)
=\mathbf{D}\left(  \mathbf{D}_{2}\left(  \gamma_{1}\right)  \right)
-\mathbf{D}\left(  \mathbf{D}_{2}\left(  \gamma_{2}\right)  \right)
\]

\end{proposition}

\begin{proof}
Let the Taylor's expansion of $\gamma_{1}$\ be
\[
\left(  x_{1},x_{2}\right)  \in\mathbb{R}^{2}\mapsto\mathbf{a+}x\mathbf{_{1}%
b}_{1}+x_{2}\mathbf{b}_{2}+x_{1}x_{2}\mathbf{b}_{12}\in\mathbb{E}%
\]
with
\[
(\mathbf{a,b}_{1},\mathbf{b}_{2},\mathbf{b}_{12})\in\mathbb{E\times E\times
E\times E}%
\]
and the Taylor's expansion of $\gamma_{2}$\ be
\[
\left(  x_{1},x_{2}\right)  \in\mathbb{R}^{2}\mapsto\mathbf{a+}x\mathbf{_{1}%
b}_{1}+x_{2}\mathbf{b}_{2}+x_{1}x_{2}\mathbf{b}_{12}^{\prime}\in\mathbb{E}%
\]
with $\mathbf{b}_{12}^{\prime}\in\mathbb{E}$ and $\mathbf{a,b}_{1}%
,\mathbf{b}_{2}$ being the same as above. Then the Taylor's expansion of
$\gamma\in\mathbb{E}\otimes\mathcal{W}_{D^{2}\oplus D}$\ with $\left(
\mathrm{id}_{\mathbb{E}}\otimes\mathcal{W}_{\varphi}\right)  \left(
\gamma\right)  =\gamma_{2}$ and $\left(  \mathrm{id}_{\mathbb{E}}%
\otimes\mathcal{W}_{\psi}\right)  \left(  \gamma\right)  =\gamma_{1}$ is
\[
\left(  x_{1},x_{2},x_{3}\right)  \in\mathbb{R}^{3}\mapsto\mathbf{a+}%
x\mathbf{_{1}b}_{1}+x_{2}\mathbf{b}_{2}+x_{1}x_{2}\mathbf{b}_{12}+x_{3}\left(
\mathbf{b}_{12}-\mathbf{b}_{12}^{\prime}\right)  \in\mathbb{E}%
\]
so that
\[
\mathbf{D}\left(  \gamma_{1}\overset{\cdot}{-}\gamma_{2}\right)
=\mathbf{b}_{12}-\mathbf{b}_{12}^{\prime}%
\]
which completes the proof.
\end{proof}

\subsection{ Differential Forms}

We recall the familiar definition.

\begin{definition}
An element $\theta$ of the space $\left[  M\otimes\mathcal{W}_{D^{n}%
}\rightarrow\mathbb{R}\right]  $ is called a (real-valued) differential
$n$-form provided that

\begin{enumerate}
\item $\theta$ is $n$-homogeneous in the sense that
\[
\theta\left(  \alpha\underset{i}{\cdot}\theta\right)  =\alpha\theta(\gamma)
\]
for any $\gamma\in M\otimes W_{D^{n}}$ and any $\alpha\in\mathbb{R}$, where
$\alpha\underset{i}{\cdot}\gamma$ is defined by
\[
\alpha\underset{i}{\cdot}\gamma=\left(  \mathrm{id}_{M}\otimes\mathcal{W}%
_{\left(  \alpha\underset{i}{\cdot}\right)  _{D^{n}}}\right)  (\gamma)
\]
with the putative mapping $\left(  \alpha\underset{i}{\cdot}\right)  _{D^{n}%
}:D^{n}\rightarrow D^{n}$ being
\[
(d_{1},...,d_{n})\in D^{n}\mapsto(d_{1},...,d_{i-1},\alpha d_{i}%
,d_{i+1},...,d_{n})\in D^{n}%
\]

\item $\theta$ is alternating in the sense that
\[
\omega(\gamma^{\sigma})=\epsilon_{\sigma}\omega(\gamma)
\]
for any $\sigma\in\mathbb{S}_{n}$, where $\mathbb{S}_{n}$ is the group of
permutations of $1,...,n$, $\epsilon_{\sigma}$ is the sign of the permutation
$\sigma$, and $\gamma^{\sigma}$ is defined by
\[
\gamma^{\sigma}=(\mathrm{id}_{M}\otimes\mathcal{W}_{\sigma_{D^{n}}})(\gamma)
\]
with the putative mapping $\sigma_{D^{n}}:D^{n}\rightarrow D^{n}$ being
\[
(d_{1},...,d_{n})\in D^{n}\mapsto(d_{\sigma(1)},...,d_{\sigma(n)})\in D^{n}%
\]

\end{enumerate}
\end{definition}

\begin{definition}
By dropping the second condition in the above definition, we get the notion of
a differential $n$-semiform on $M$.
\end{definition}

\begin{notation}
We denote by $\Omega^{n}\left(  M\right)  $ and $\widetilde{\Omega}^{n}\left(
M\right)  $\ the totality of differential $n$-forms on $M$ and that of
differential $n$-semiforms on $M$, respectively. We denote by $\Omega\left(
M\right)  $ and $\widetilde{\Omega}\left(  M\right)  $\ the totality of
differential forms on $M$ and that of differential semiforms on $M$, respectively.
\end{notation}

\begin{definition}
Given $\theta_{1}\in\left[  M\otimes\mathcal{W}_{D^{p}}\rightarrow
\mathbb{R}\right]  $ and $\theta_{2}\in\left[  M\otimes\mathcal{W}_{D^{q}%
}\rightarrow\mathbb{R}\right]  $, we define $\theta_{1}\otimes\theta_{2}%
\in\left[  M\otimes\mathcal{W}_{D^{p+q}}\rightarrow\mathbb{R}\right]  $ to be
\begin{align*}
&  \left(  \theta_{1}\otimes\theta_{2}\right)  \left(  \gamma\right) \\
&  =\theta_{1}\left(  \mathcal{W}_{\left(  d_{1},...,d_{p}\right)  \in
D^{p}\mapsto\left(  d_{1},...,d_{p},0,...,0\right)  \in D^{p+q}}\left(
\gamma\right)  \right)  \theta_{2}\left(  \mathcal{W}_{\left(  d_{1}%
,...,d_{q}\right)  \in D^{q}\mapsto\left(  0,...,0,d_{1},...,d_{q}\right)  \in
D^{p+q}}\left(  \gamma\right)  \right)
\end{align*}

\end{definition}

It is easy to see the following.

\begin{proposition}
\label{t2.5.1}If $\theta_{1}$ is a differential $p$-semiform on $M$ and
$\theta_{2}$ is a differential $q$-semiform on $M$, then $\theta_{1}%
\otimes\theta_{2}$ is a differential $\left(  p+q\right)  $-semiform on $M$.
\end{proposition}

\begin{proposition}
\label{t2.5.1'}Given $\theta_{1}\in\left[  M\otimes\mathcal{W}_{D^{p}%
}\rightarrow\mathbb{R}\right]  $, $\theta_{2}\in\left[  M\otimes
\mathcal{W}_{D^{q}}\rightarrow\mathbb{R}\right]  $ and $\theta_{3}\in\left[
M\otimes\mathcal{W}_{D^{r}}\rightarrow\mathbb{R}\right]  $, we have
\[
\left(  \theta_{1}\otimes\theta_{2}\right)  \otimes\theta_{3}=\theta
_{1}\otimes\left(  \theta_{2}\otimes\theta_{3}\right)
\]

\end{proposition}

\begin{remark}
Therefore we can write $\theta_{1}\otimes\theta_{2}\otimes\theta_{3}$ without ambiguity.
\end{remark}

\begin{definition}
Given $\theta\in\left[  M\otimes\mathcal{W}_{D^{p}}\rightarrow\mathbb{R}%
\right]  $, we define $\mathcal{A}\theta\in\left[  M\otimes\mathcal{W}_{D^{p}%
}\rightarrow\mathbb{R}\right]  $\ to be
\[
\mathcal{A}\theta=\sum_{\sigma\in\mathbb{S}_{p}}\varepsilon_{\sigma}%
\theta^{\sigma}%
\]

\end{definition}

\begin{notation}
Given $\theta\in\left[  M\otimes\mathcal{W}_{D^{p+q}}\rightarrow
\mathbb{R}\right]  $, we write $\mathcal{A}_{p,q}\theta$ for
$(1/p!q!)\mathcal{A}\theta$. Given $\theta\in\left[  M\otimes\mathcal{W}%
_{D^{p+q+r}}\rightarrow\mathbb{R}\right]  $, we write $\mathcal{A}%
_{p,q,r}\theta$ for $(1/p!q!r!)\mathcal{A}\theta$.
\end{notation}

\begin{definition}
Given $\theta_{1}\in\left[  M\otimes\mathcal{W}_{D^{p}}\rightarrow
\mathbb{R}\right]  $ and $\theta_{2}\in\left[  M\otimes\mathcal{W}_{D^{q}%
}\rightarrow\mathbb{R}\right]  $, we define $\theta_{1}\wedge\theta_{2}%
\in\left[  M\otimes\mathcal{W}_{D^{p+q}}\rightarrow\mathbb{R}\right]  $ to be
$\mathcal{A}_{p,q}\left(  \theta_{1}\otimes\theta_{2}\right)  $.
\end{definition}

It is easy to see the following.

\begin{proposition}
\label{t2.5.2}If $\theta$ is a differential semiform on $M$, then
$\mathcal{A}\theta$ is a differential form.
\end{proposition}

\begin{proposition}
\label{t2.5.2'}Given $\theta_{1}\in\left[  M\otimes\mathcal{W}_{D^{p}%
}\rightarrow\mathbb{R}\right]  $, $\theta_{2}\in\left[  M\otimes
\mathcal{W}_{D^{q}}\rightarrow\mathbb{R}\right]  $ and $\theta_{3}\in\left[
M\otimes\mathcal{W}_{D^{r}}\rightarrow\mathbb{R}\right]  $, we have
\begin{align*}
&  \mathcal{A}_{p,q+r}\left(  \theta_{1}\otimes\mathcal{A}_{q,r}\left(
\theta_{2}\otimes\theta_{3}\right)  \right) \\
&  =\mathcal{A}_{p+q,r}\left(  \mathcal{A}_{p,q}\left(  \theta_{1}%
\otimes\theta_{2}\right)  \otimes\theta_{3}\right) \\
&  =\mathcal{A}_{p,q,r}\left(  \theta_{1}\otimes\theta_{2}\otimes\theta
_{3}\right)
\end{align*}

\end{proposition}

\begin{corollary}
Given $\theta_{1}\in\left[  M\otimes\mathcal{W}_{D^{p}}\rightarrow
\mathbb{R}\right]  $, $\theta_{2}\in\left[  M\otimes\mathcal{W}_{D^{q}%
}\rightarrow\mathbb{R}\right]  $ and $\theta_{3}\in\left[  M\otimes
\mathcal{W}_{D^{r}}\rightarrow\mathbb{R}\right]  $, we have
\[
\left(  \theta_{1}\wedge\theta_{2}\right)  \wedge\theta_{3}=\theta_{1}%
\wedge\left(  \theta_{2}\wedge\theta_{3}\right)
\]

\end{corollary}

It is easy to see the following two propositions.

\begin{proposition}
\label{t2.5.3}Convenient vector spaces are Euclidean vector spaces which are
microlinear and Weil exponentiable.
\end{proposition}

\begin{proposition}
\label{t2.5.4}The spaces $\Omega^{n}\left(  M\right)  $\ and $\widetilde
{\Omega}^{n}\left(  M\right)  $\ are convenient vector spaces.
\end{proposition}

Therefore we have

\begin{proposition}
\label{t2.5.5}The spaces $\Omega^{n}\left(  M\right)  $\ and $\widetilde
{\Omega}^{n}\left(  M\right)  $\ are Euclidean vector spaces which are microlinear.
\end{proposition}

\section{Tangent-Vector-Valued Differential Forms}

Our two distinct but equivalent viewpoints of tangent-vector-valued
differential forms on $M$ are based upon the following exponential law:
\begin{align*}
&  \left[  M\otimes\mathcal{W}_{D^{p}}\rightarrow M\otimes\mathcal{W}%
_{D}\right] \\
&  =\left[  M\otimes\mathcal{W}_{D^{p}}\rightarrow M\right]  \otimes
\mathcal{W}_{D}%
\end{align*}
If $p=0$, the above law degenerates into the corresponding one in \S \ref{2.1}.

The first viewpoint, which is highly orthodox, goes as follows.

\begin{definition}
A tangent-vector-valued $p$\textit{-form on} $M$ is a mapping $\xi
:M\otimes\mathcal{W}_{D^{p}}\rightarrow M\otimes\mathcal{W}_{D}$ subject to
the following three conditions:

\begin{enumerate}
\item We have
\[
\pi_{M}^{M\otimes\mathcal{W}_{D^{p}}}(\gamma)=\pi_{M}^{M\otimes\mathcal{W}%
_{D}}(\xi(\gamma))
\]
for any $\gamma\in M\otimes\mathcal{W}_{D^{p}}$.

\item We have
\[
\xi(\alpha\underset{i}{\cdot}\gamma)=\alpha\xi(\gamma)
\]
for any $\alpha\in\mathbb{R}$, any $\gamma\in M\otimes\mathcal{W}_{D^{p}} $
and any natural number $i$ with $1\leq i\leq p$.

\item We have
\[
\xi(\gamma^{\sigma})=\varepsilon_{\sigma}\xi(\gamma)
\]
for any $\gamma\in M\otimes\mathcal{W}_{D^{p}}$ and any $\sigma\in
\mathbb{S}_{p}$.
\end{enumerate}

By dropping the third condition, we get the weaker notion of a
tangent-vector-valued $p$\textit{-semiform on} $M$.
\end{definition}

The other viewpoint, which is highly radical, goes as follows.

\begin{definition}
A tangent-vector-valued $p$\textit{-form on} $M$ is an element $\xi\in\left[
M\otimes\mathcal{W}_{D^{p}}\rightarrow M\right]  \otimes\mathcal{W}_{D}$
pursuant to the following three conditions:

\begin{enumerate}
\item We have
\[
\pi_{\left[  M\otimes\mathcal{W}_{D^{p}}\rightarrow M\right]  }^{\left[
M\otimes\mathcal{W}_{D^{p}}\rightarrow M\right]  \otimes\mathcal{W}_{D}%
}\left(  \xi\right)  =\delta_{M}^{p}%
\]
where $\pi_{\left[  M\otimes\mathcal{W}_{D^{p}}\rightarrow M\right]
}^{\left[  M\otimes\mathcal{W}_{D^{p}}\rightarrow M\right]  \otimes
\mathcal{W}_{D}}:\left[  M\otimes\mathcal{W}_{D^{p}}\rightarrow M\right]
\otimes\mathcal{W}_{D}\rightarrow\left[  M\otimes\mathcal{W}_{D^{p}%
}\rightarrow M\right]  $ is the canonical projection, and $\delta_{M}^{p}$,
called a ($p$-dimensional) Dirac distribution on $M$, denotes the canonical
projection $\pi_{M}^{M\otimes\mathcal{W}_{D^{p}}}:M\otimes\mathcal{W}_{D^{p}%
}\rightarrow M$.

\item We have
\[
\left(  \left(  \left(  \alpha\underset{i}{\cdot}\right)  _{M\otimes
\mathcal{W}_{D^{p}}}\right)  ^{\ast}\otimes\mathrm{id}_{\mathcal{W}_{D}%
}\right)  \left(  \xi\right)  =\alpha\xi
\]

\item We have
\[
\left(  \left(  \left(  \cdot^{\sigma}\right)  _{M\otimes\mathcal{W}_{D^{p}}%
}\right)  ^{\ast}\otimes\mathrm{id}_{\mathcal{W}_{D}}\right)  \left(
\xi\right)  =\varepsilon_{\sigma}\xi
\]

\end{enumerate}

By dropping the third condition, we get the weaker notion of a
tangent-vector-valued $p$\textit{-semiform on} $M$.
\end{definition}

The following proposition is simple but very important and highly useful.

\begin{proposition}
\label{t3.1}The addition for tangent-vector-valued\textit{\ }$p$%
\textit{-semiforms on} $M$ in the first sense (i.e., using the fiberwise
addition of the vector bundle $M\otimes\mathcal{W}_{D}\rightarrow M$) and that
in the second sense (i.e., as the addition of tangent vectors to the space
$[M\otimes\mathcal{W}_{D^{p}}\rightarrow M]$ at $\delta_{M}^{p}$) coincide.
\end{proposition}

\begin{proof}
This follows mainly from the following exponential law:
\begin{align*}
\lbrack M\otimes\mathcal{W}_{D^{p}}  &  \rightarrow M\otimes\mathcal{W}%
_{D(2)}]\\
&  =[M\otimes\mathcal{W}_{D^{p}}\rightarrow M]\otimes\mathcal{W}_{D(2)}%
\end{align*}
The details can safely be left to the reader.
\end{proof}

Unless stated to the contrary, we will use the terms
\textit{tangent-vector-valued }$p$\textit{-semiforms on} $M$ and
\textit{tangent-vector-valued }$p$\textit{-forms on} $M$ in the second sense.

\section{The Fr\"{o}licher-Nijenhuis Bracket}

\subsection{The Jacobi Identity for the Lie Bracket}

Let us begin this subsection with the following definition.

\begin{definition}
Given $\eta_{1}\in\left[  M\otimes\mathcal{W}_{D^{p}}\rightarrow M\right]  $
and $\eta_{2}\in\left[  M\otimes\mathcal{W}_{D^{q}}\rightarrow M\right]  $,
two kinds of convolution, both of which belong to $\left[  M\otimes
\mathcal{W}_{D^{p+q}}\rightarrow M\right]  $, are defined. The first, to be
denoted by $\eta_{1}\ast\eta_{2}$, is defined to be
\begin{align*}
M\otimes\mathcal{W}_{D^{p+q}}  &  =M\otimes\left(  \mathcal{W}_{D^{p}}%
\otimes_{\infty}\mathcal{W}_{D^{q}}\right)  =M\otimes\left(  \mathcal{W}%
_{D^{q}}\otimes_{\infty}\mathcal{W}_{D^{p}}\right) \\
&  =\left(  M\otimes\mathcal{W}_{D^{q}}\right)  \otimes\mathcal{W}_{D^{p}}
\begin{array}
[c]{c}%
\eta_{2}\otimes\mathrm{id}_{\mathcal{W}_{D^{p}}}\\
\rightarrow\\
\,
\end{array}
M\otimes\mathcal{W}_{D^{p}}
\begin{array}
[c]{c}%
\eta_{1}\\
\rightarrow\\
\,
\end{array}
M
\end{align*}
The second, to be denoted by $\eta_{1}\widetilde{\ast}\eta_{2}$, is defined to
be
\begin{align*}
M\otimes\mathcal{W}_{D^{p+q}}  &  =M\otimes\left(  \mathcal{W}_{D^{p}}%
\otimes_{\infty}\mathcal{W}_{D^{q}}\right) \\
&  =\left(  M\otimes\mathcal{W}_{D^{p}}\right)  \otimes\mathcal{W}_{D^{q}}
\begin{array}
[c]{c}%
\eta_{1}\otimes\mathrm{id}_{\mathcal{W}_{D^{q}}}\\
\rightarrow\\
\,
\end{array}
M\otimes\mathcal{W}_{D^{q}}
\begin{array}
[c]{c}%
\eta_{2}\\
\rightarrow\\
\,
\end{array}
M
\end{align*}

\end{definition}

\begin{remark}
\thinspace

\begin{enumerate}
\item Our two convolutions are reminiscent of the familiar ones in abstract
harmonic analysis and the theory of Schwartz distributions.

\item If $p=q=0$, then
\[
M\otimes\mathcal{W}_{D^{p}}=M\otimes\mathcal{W}_{D^{q}}=M\otimes
\mathcal{W}_{D^{p+q}}=M
\]
so that
\[
\left[  M\otimes\mathcal{W}_{D^{p}}\rightarrow M\right]  =\left[
M\otimes\mathcal{W}_{D^{q}}\rightarrow M\right]  =\left[  M\otimes
\mathcal{W}_{D^{p+q}}\rightarrow M\right]  =[M\rightarrow M]
\]
in which we have
\begin{align*}
\eta_{1}\ast\eta_{2}  &  =\eta_{1}\circ\eta_{2}\\
\eta_{1}\widetilde{\ast}\eta_{2}  &  =\eta_{2}\circ\eta_{1}%
\end{align*}

\end{enumerate}
\end{remark}

\begin{notation}
Given $\sigma\in\mathbb{S}_{p}$ and $\eta\in\left[  M\otimes\mathcal{W}%
_{D^{p}}\rightarrow M\right]  $, we let $\eta^{\sigma}$ denote
\[
\eta\circ\left(  \mathrm{id}_{M}\otimes\mathcal{W}_{\left(  d_{1}%
,...,d_{p}\right)  \in D^{p}\mapsto\left(  d_{\sigma\left(  1\right)
},...,d_{\sigma\left(  p\right)  }\right)  \in D^{p}}\right)
\]

\end{notation}

It should be obvious that

\begin{proposition}
\label{t4.1.1}Given $\eta_{1}\in\left[  M\otimes\mathcal{W}_{D^{p}}\rightarrow
M\right]  $ and $\eta_{2}\in\left[  M\otimes\mathcal{W}_{D^{q}}\rightarrow
M\right]  $, we have
\begin{align*}
\left(  \eta_{2}\ast\eta_{1}\right)  ^{\sigma_{p,q}}  &  =\eta_{1}%
\widetilde{\ast}\eta_{2}\\
\left(  \eta_{2}\widetilde{\ast}\eta_{1}\right)  ^{\sigma_{p,q}}  &  =\eta
_{1}\ast\eta_{2}%
\end{align*}
where $\sigma_{p,q}$ is the permutation mapping the sequence
$1,...,q,q+1,...,p+q$ to the sequence $q+1,...,p+q,1,...,q$, namely,
\[
\sigma_{p,q}=\left(
\begin{array}
[c]{cccccc}%
1 & ... & p & p+1 & ... & p+q\\
q+1 & ... & p+q & 1 & ... & q
\end{array}
\right)
\]

\end{proposition}

It should also be obvious that

\begin{proposition}
\label{t4.1.2}Given $\eta_{1}\in\left[  M\otimes\mathcal{W}_{D^{p}}\rightarrow
M\right]  $, $\eta_{2}\in\left[  M\otimes\mathcal{W}_{D^{q}}\rightarrow
M\right]  $ and $\eta_{3}\in\left[  M\otimes\mathcal{W}_{D^{r}}\rightarrow
M\right]  $, we have
\begin{align*}
(\eta_{1}\ast\eta_{2})\ast\eta_{3}  &  =\eta_{1}\ast(\eta_{2}\ast\eta_{3})\\
(\eta_{1}\widetilde{\ast}\eta_{2})\widetilde{\ast}\eta_{3}  &  =\eta
_{1}\widetilde{\ast}(\eta_{2}\widetilde{\ast}\eta_{3})
\end{align*}

\end{proposition}

\begin{remark}
This proposition enables us to write, e.g., $\eta_{1}\ast\eta_{2}\ast\eta_{3}
$ without parentheses in place of $(\eta_{1}\ast\eta_{2})\ast\eta_{3}$ or
$\eta_{1}\ast(\eta_{2}\ast\eta_{3})$. Similarly for $\eta_{1}\widetilde{\ast
}\eta_{2}\widetilde{\ast}\eta_{3}$.
\end{remark}

\begin{definition}
The canonical projection $\pi_{M}^{M\otimes\mathcal{W}_{D^{p}}}:M\otimes
\mathcal{W}_{D^{p}}\rightarrow M$ is called a ($p$-dimensional) Dirac
distribution on $M$, which is to be denoted by $\delta_{M}^{p}$
\end{definition}

The following proposition should be obvious.

\begin{proposition}
\label{t4.1.3}If one of $\eta_{1}\in\left[  M\otimes\mathcal{W}_{D^{p}%
}\rightarrow M\right]  $ and $\eta_{2}\in\left[  M\otimes\mathcal{W}_{D^{q}%
}\rightarrow M\right]  $ is a Dirac distribution, then $\eta_{1}\ast\eta_{2}$
and $\eta_{1}\widetilde{\ast}\eta_{2}$ coincide. In particular, if both of
$\eta_{1}$ and $\eta_{2}$ are Dirac distributions, then $\eta_{1}\ast\eta
_{2}=\eta_{1}\widetilde{\ast}\eta_{2}$ is also a Dirac distribution.
\end{proposition}

\begin{definition}
An element $\xi\in\left[  M\otimes\mathcal{W}_{D^{p}}\rightarrow M\right]
\otimes\mathcal{W}_{D^{n}}$ with
\[
\pi_{\left[  M\otimes\mathcal{W}_{D^{p}}\rightarrow M\right]  }^{\left[
M\otimes\mathcal{W}_{D^{p}}\rightarrow M\right]  \otimes\mathcal{W}_{D^{n}}%
}\left(  \xi\right)  =\delta_{M}^{p}%
\]
is called an $(n,p)$-icon on $M$.
\end{definition}

\begin{remark}
By dropping the second and third conditions in the second definition of a
tangent-vector-valued $p$\textit{-form on} $M$ in the preceding section, we
rediscover the notion of a $(1,p)$-icon on $M$.
\end{remark}

\begin{definition}
We define a binary mapping
\begin{align*}
\circledast &  :\left(  \left[  M\otimes\mathcal{W}_{D^{p}}\rightarrow
M\right]  \otimes\mathcal{W}_{D^{m}}\right)  \times\left(  \left[
M\otimes\mathcal{W}_{D^{q}}\rightarrow M\right]  \otimes\mathcal{W}_{D^{n}%
}\right)  \rightarrow\\
&  \left[  M\otimes\mathcal{W}_{D^{p+q}}\rightarrow M\right]  \otimes
\mathcal{W}_{D^{m+n}}%
\end{align*}
to be
\begin{align*}
&  \left(  \left[  M\otimes\mathcal{W}_{D^{p}}\rightarrow M\right]
\otimes\mathcal{W}_{D^{m}}\right)  \times\left(  \left[  M\otimes
\mathcal{W}_{D^{q}}\rightarrow M\right]  \otimes\mathcal{W}_{D^{n}}\right) \\
&  \underline{\left(  \mathrm{id}_{\left[  M\otimes\mathcal{W}_{D^{p}%
}\rightarrow M\right]  }\otimes\mathcal{W}_{\left(  d_{1},...,d_{m}%
,d_{m+1},...,d_{m+n}\right)  \in D^{m+n}\mapsto\left(  d_{1},...,d_{m}\right)
\in D^{m}}\right)  \times}\\
&  \underrightarrow{\left(  \mathrm{id}_{\left[  M\otimes\mathcal{W}_{D^{q}%
}\rightarrow M\right]  }\otimes\mathcal{W}_{\left(  d_{1},...,d_{m}%
,d_{m+1},...,d_{m+n}\right)  \in D^{m+n}\mapsto\left(  d_{m+1},...,d_{m+n}%
\right)  \in D^{n}}\right)  }\\
&  \left(  \left[  M\otimes\mathcal{W}_{D^{p}}\rightarrow M\right]
\otimes\mathcal{W}_{D^{m+n}}\right)  \times\left(  \left[  M\otimes
\mathcal{W}_{D^{q}}\rightarrow M\right]  \otimes\mathcal{W}_{D^{m+n}}\right)
\\
&  =\left(  \left[  M\otimes\mathcal{W}_{D^{p}}\rightarrow M\right]
\times\left[  M\otimes\mathcal{W}_{D^{q}}\rightarrow M\right]  \right)
\otimes\mathcal{W}_{D^{m+n}}\\
&  \underrightarrow{\ast\otimes\mathrm{id}_{\mathcal{W}_{D^{m+n}}}}\left[
M\otimes\mathcal{W}_{D^{p+q}}\rightarrow M\right]  \otimes\mathcal{W}%
_{D^{m+n}}%
\end{align*}

\end{definition}

\begin{definition}
We define a binary mapping
\begin{align*}
\widetilde{\circledast}  &  :\left(  \left[  M\otimes\mathcal{W}_{D^{p}%
}\rightarrow M\right]  \otimes\mathcal{W}_{D^{m}}\right)  \times\left(
\left[  M\otimes\mathcal{W}_{D^{q}}\rightarrow M\right]  \otimes
\mathcal{W}_{D^{n}}\right)  \rightarrow\\
&  \left[  M\otimes\mathcal{W}_{D^{p+q}}\rightarrow M\right]  \otimes
\mathcal{W}_{D^{m+n}}%
\end{align*}
to be
\begin{align*}
&  \left(  \left[  M\otimes\mathcal{W}_{D^{p}}\rightarrow M\right]
\otimes\mathcal{W}_{D^{m}}\right)  \times\left(  \left[  M\otimes
\mathcal{W}_{D^{q}}\rightarrow M\right]  \otimes\mathcal{W}_{D^{n}}\right) \\
&  \underline{\left(  \mathrm{id}_{\left[  M\otimes\mathcal{W}_{D^{p}%
}\rightarrow M\right]  }\otimes\mathcal{W}_{\left(  d_{1},...,d_{m}%
,d_{m+1},...,d_{m+n}\right)  \in D^{m+n}\mapsto\left(  d_{1},...,d_{m}\right)
\in D^{m}}\right)  \times}\\
&  \underrightarrow{\left(  \mathrm{id}_{\left[  M\otimes\mathcal{W}_{D^{q}%
}\rightarrow M\right]  }\otimes\mathcal{W}_{\left(  d_{1},...,d_{m}%
,d_{m+1},...,d_{m+n}\right)  \in D^{m+n}\mapsto\left(  d_{m+1},...,d_{m+n}%
\right)  \in D^{n}}\right)  }\\
&  \left(  \left[  M\otimes\mathcal{W}_{D^{p}}\rightarrow M\right]
\otimes\mathcal{W}_{D^{m+n}}\right)  \times\left(  \left[  M\otimes
\mathcal{W}_{D^{q}}\rightarrow M\right]  \otimes\mathcal{W}_{D^{m+n}}\right)
\\
&  =\left(  \left[  M\otimes\mathcal{W}_{D^{p}}\rightarrow M\right]
\times\left[  M\otimes\mathcal{W}_{D^{q}}\rightarrow M\right]  \right)
\otimes\mathcal{W}_{D^{m+n}}\\
&  \underrightarrow{\widetilde{\ast}\otimes\mathrm{id}_{\mathcal{W}_{D^{m+n}}%
}}\left[  M\otimes\mathcal{W}_{D^{p+q}}\rightarrow M\right]  \otimes
\mathcal{W}_{D^{m+n}}%
\end{align*}

\end{definition}

\begin{proposition}
\label{t4.1.4}Given $\xi_{1}\in\left[  M\otimes\mathcal{W}_{D^{p}}\rightarrow
M\right]  \otimes\mathcal{W}_{D^{l}}$, $\xi_{2}\in\left[  M\otimes
\mathcal{W}_{D^{q}}\rightarrow M\right]  \otimes\mathcal{W}_{D^{m}}$ and
$\xi_{3}\in\left[  M\otimes\mathcal{W}_{D^{r}}\rightarrow M\right]
\otimes\mathcal{W}_{D^{n}}$, we have
\begin{align*}
(\xi_{1}\circledast\xi_{2})\circledast\xi_{3}  &  =\xi_{1}\circledast(\xi
_{2}\circledast\xi_{3})\\
(\xi_{1}\widetilde{\circledast}\xi_{2})\widetilde{\circledast}\xi_{3}  &
=\xi_{1}\widetilde{\circledast}(\xi_{2}\widetilde{\circledast}\xi_{3})
\end{align*}

\end{proposition}

It should be obvious that

\begin{lemma}
\label{t4.1.5}For any $(1,p)$-icon $\xi_{1}$\ on $M$ and any $(1,q)$-icon
$\xi_{2}$\ on $M$, we have
\begin{align*}
&  \left(  \mathrm{id}_{\left[  M\otimes\mathcal{W}_{D^{p}}\rightarrow
M\right]  }\otimes\mathcal{W}_{\left(  d_{1},d_{2}\right)  \in D(2)\mapsto
\left(  d_{1},d_{2}\right)  \in D^{2}}\right)  \left(  \xi_{1}\circledast
\xi_{2}\right) \\
&  =\left(  \mathrm{id}_{\left[  M\otimes\mathcal{W}_{D^{p}}\rightarrow
M\right]  }\otimes\mathcal{W}_{\left(  d_{1},d_{2}\right)  \in D(2)\mapsto
\left(  d_{1},d_{2}\right)  \in D^{2}}\right)  \left(  \xi_{1}\widetilde
{\circledast}\xi_{2}\right)
\end{align*}

\end{lemma}

Therefore the following definition is meaningful.

\begin{definition}
For any $(1,p)$-icon $\xi_{1}$\ on $M$ and any $(1,q)$-icon $\xi_{2}$\ on $M$,
their Lie bracket $\left[  \xi_{1},\xi_{2}\right]  _{L}\in\left[
M\otimes\mathcal{W}_{D^{p+q}}\rightarrow M\right]  \otimes\mathcal{W}_{D}$ is
defined to be
\[
\left[  \xi_{1},\xi_{2}\right]  _{L}=\xi_{1}\widetilde{\circledast}\xi
_{2}\overset{\cdot}{-}\xi_{1}\circledast\xi_{2}%
\]

\end{definition}

It is easy to see that

\begin{lemma}
\label{t4.1.6}In the above definition, $\left[  \xi_{1},\xi_{2}\right]  _{L}$
is always a $(1,p+q)$-icon on $M$.
\end{lemma}

\begin{proposition}
\label{t4.1.7}If $\xi_{1}$ is a tangent-vector-valued $p$\textit{-semiform on}
$M$ and $\xi_{2}$ is a tangent-vector-valued $q$\textit{-semiform on} $M $,
then we have
\begin{align*}
&  \left(  \left(  \left(  \alpha\underset{i}{\cdot}\right)  _{M\otimes
\mathcal{W}_{D^{p+q}}}\right)  ^{\ast}\otimes\mathrm{id}_{\mathcal{W}_{D^{2}}%
}\right)  (\xi_{1}\circledast\xi_{2})\\
&  =\left(  \mathrm{id}_{\left[  M\otimes\mathcal{W}_{D^{p+q}}\rightarrow
M\right]  }\otimes\mathcal{W}_{\left(  d_{1},d_{2}\right)  \in D^{2}%
\mapsto(\alpha d_{1},d_{2})\in D^{2}}\right)  (\xi_{1}\circledast\xi_{2})\\
&  \left(  \left(  \left(  \alpha\underset{i}{\cdot}\right)  _{M\otimes
\mathcal{W}_{D^{p+q}}}\right)  ^{\ast}\otimes\mathrm{id}_{\mathcal{W}_{D^{2}}%
}\right)  (\xi_{1}\widetilde{\circledast}\xi_{2})\\
&  =\left(  \mathrm{id}_{\left[  M\otimes\mathcal{W}_{D^{p+q}}\rightarrow
M\right]  }\otimes\mathcal{W}_{\left(  d_{1},d_{2}\right)  \in D^{2}%
\mapsto(\alpha d_{1},d_{2})\in D^{2}}\right)  (\xi_{1}\widetilde{\circledast
}\xi_{2})
\end{align*}
for any natural number $i$ with $1\leq i\leq p$, while we have
\begin{align*}
&  \left(  \left(  \left(  \alpha\underset{i}{\cdot}\right)  _{M\otimes
\mathcal{W}_{D^{p+q}}}\right)  ^{\ast}\otimes\mathrm{id}_{\mathcal{W}_{D^{2}}%
}\right)  (\xi_{1}\circledast\xi_{2})\\
&  =\left(  \mathrm{id}_{\left[  M\otimes\mathcal{W}_{D^{p+q}}\rightarrow
M\right]  }\otimes\mathcal{W}_{\left(  d_{1},d_{2}\right)  \in D^{2}%
\mapsto(d_{1},\alpha d_{2})\in D^{2}}\right)  (\xi_{1}\circledast\xi_{2})\\
&  \left(  \left(  \left(  \alpha\underset{i}{\cdot}\right)  _{M\otimes
\mathcal{W}_{D^{p+q}}}\right)  ^{\ast}\otimes\mathrm{id}_{\mathcal{W}_{D^{2}}%
}\right)  (\xi_{1}\widetilde{\circledast}\xi_{2})\\
&  =\left(  \mathrm{id}_{\left[  M\otimes\mathcal{W}_{D^{p+q}}\rightarrow
M\right]  }\otimes\mathcal{W}_{\left(  d_{1},d_{2}\right)  \in D^{2}%
\mapsto(d_{1},\alpha d_{2})\in D^{2}}\right)  (\xi_{1}\widetilde{\circledast
}\xi_{2})
\end{align*}
for any natural number $i$ with $p+1\leq i\leq p+q$.
\end{proposition}

\begin{corollary}
\label{t4.1.8}If $\xi_{1}$ is a tangent-vector-valued $p$\textit{-semiform on}
$M$ and $\xi_{2}$ is a tangent-vector-valued $q$\textit{-semiform on} $M $,
then $\left[  \xi_{1},\xi_{2}\right]  _{L}$ is a tangent-vector-valued $(p+q)
$\textit{-semiform on} $M$.
\end{corollary}

\begin{proof}
It suffices to see that
\begin{align*}
&  \left(  \left(  \left(  \alpha\underset{i}{\cdot}\right)  _{M\otimes
\mathcal{W}_{D^{p+q}}}\right)  ^{\ast}\otimes\mathrm{id}_{\mathcal{W}_{D}%
}\right)  \left(  \left[  \xi_{1},\xi_{2}\right]  _{L}\right) \\
&  =\left(  \mathrm{id}_{\left[  M\otimes\mathcal{W}_{D^{p+q}}\rightarrow
M\right]  }\otimes\mathcal{W}_{d\in D\mapsto\alpha d\in D}\right)  \left(
\left[  \xi_{1},\xi_{2}\right]  _{L}\right)
\end{align*}
for any $\alpha\in\mathbb{R}$ and any natural number $i$ with $1\leq i\leq
p+q$, which follows easily from the \ above Proposition and Proposition 5 in
\S 3.4 of Lavendhomme \cite{lav}.
\end{proof}

\begin{proposition}
\label{t4.1.12}If $\xi_{1},\xi_{1}^{\prime}$ are tangent-vector-valued
$p$\textit{-semiforms on} $M$ and $\xi_{2},\xi_{2}^{\prime}$ are
tangent-vector-valued $q$\textit{-semiforms on} $M$ with $\alpha\in\mathbb{R}
$, then we have the following:

\begin{enumerate}
\item
\[
\left[  \alpha\xi_{1},\xi_{2}\right]  _{L}=\alpha\left[  \xi_{1},\xi
_{2}\right]  _{L}%
\]

\item
\[
\left[  \xi_{1}+\xi_{1}^{\prime},\xi_{2}\right]  _{L}=\left[  \xi_{1},\xi
_{2}\right]  _{L}+\left[  \xi_{1}^{\prime},\xi_{2}\right]  _{L}%
\]

\item
\[
\left[  \xi_{1},\alpha\xi_{2}\right]  _{L}=\alpha\left[  \xi_{1},\xi
_{2}\right]  _{L}%
\]

\item
\[
\left[  \xi_{1},\xi_{2}+\xi_{2}^{\prime}\right]  _{L}=\left[  \xi_{1},\xi
_{2}\right]  _{L}+\left[  \xi_{1},\xi_{2}^{\prime}\right]  _{L}%
\]

\end{enumerate}
\end{proposition}

\begin{proof}
The statements 1 and 3 follow from Proposition 5 in \S 3.4 of Lavendhomme
\cite{lav}, while the statements 2 and 4 follow from the statements 1 and 3 respectively.
\end{proof}

\begin{notation}
Given $\xi\in\left[  M\otimes\mathcal{W}_{D^{p}}\rightarrow M\right]
\otimes\mathcal{W}_{D^{n}}$ and $\sigma\in\mathbb{S}_{p}$, $\xi^{\sigma} $
denotes
\[
\left(  \left(  \,\right)  _{\left[  M\otimes\mathcal{W}_{D^{p}}\rightarrow
M\right]  }^{\sigma}\otimes\mathrm{id}_{\mathcal{W}_{D^{n}}}\right)  \left(
\xi\right)
\]
where $\left(  \,\right)  _{\left[  M\otimes\mathcal{W}_{D^{p}}\rightarrow
M\right]  }^{\sigma}:\left[  M\otimes\mathcal{W}_{D^{p}}\rightarrow M\right]
\rightarrow\left[  M\otimes\mathcal{W}_{D^{p}}\rightarrow M\right]  $ denotes
the operation
\[
\eta\in\left[  M\otimes\mathcal{W}_{D^{p}}\rightarrow M\right]  \mapsto
\eta\circ\left(  \mathrm{id}_{M}\otimes\mathcal{W}_{\left(  d_{1}%
,...,d_{p}\right)  \in D^{p}\mapsto\left(  d_{\sigma\left(  1\right)
},...,d_{\sigma\left(  p\right)  }\right)  \in D^{p}}\right)
\]

\end{notation}

We will show that the Lie bracket $\left[  \,\right]  _{L}$ is antisymmetric.

\begin{proposition}
\label{t4.1.9}Let $\xi_{1}$ be a $(1,p)$-icon on $M$ and $\xi_{2}$ a $(1,q)
$-icon on $M$. Then we have the following antisymmetry:
\[
\left[  \xi_{1},\xi_{2}\right]  _{L}+\left(  \left[  \xi_{2},\xi_{1}\right]
_{L}\right)  ^{\sigma_{p,q}}=0
\]

\end{proposition}

\begin{proof}
This follows from Propositions 4 and 6 in \S 3.4 of Lavendhomme \cite{lav}.
More specifically we have
\begin{align*}
&  \left[  \xi_{1},\xi_{2}\right]  _{L}+\left(  \left[  \xi_{2},\xi
_{1}\right]  _{L}\right)  ^{\sigma_{p,q}}\\
&  =(\xi_{1}\widetilde{\circledast}\xi_{2}\overset{\cdot}{-}\xi_{1}%
\circledast\xi_{2})+(\left(  \xi_{2}\widetilde{\circledast}\xi_{1}\right)
^{\sigma_{p,q}}\overset{\cdot}{-}\left(  \xi_{2}\circledast\xi_{1}\right)
^{\sigma_{p,q}})\\
&  =(\xi_{1}\widetilde{\circledast}\xi_{2}\overset{\cdot}{-}\xi_{1}%
\circledast\xi_{2})+\left(  \xi_{1}\circledast\xi_{2}\overset{\cdot}{-}\xi
_{1}\widetilde{\circledast}\xi_{2}\right) \\
&  \text{[By Proposition \ref{t4.1.1}]}\\
&  =0\text{ \ }%
\end{align*}

\end{proof}

\begin{theorem}
\label{t4.1.10}Let $\xi_{1}$ be a $(1,p)$-icon on $M$, $\xi_{2}$ a
$(1,q)$-icon on $M$, and $\xi_{3}$ a $(1,r)$-icon on $M$. Then we have the
following Jacobi identity:
\[
\left[  \xi_{1},\left[  \xi_{2},\xi_{3}\right]  _{L}\right]  _{L}+\left(
\left[  \xi_{2},\left[  \xi_{3},\xi_{1}\right]  _{L}\right]  _{L}\right)
^{\sigma_{p,q+r}}+\left(  \left[  \xi_{3},\left[  \xi_{1},\xi_{2}\right]
_{L}\right]  _{L}\right)  ^{\sigma_{r,p+q}}=0
\]

\end{theorem}

In order to establish the above theorem, we need the following simple lemma,
which is a tiny generalization of Proposition 2.6 of \cite{nishi-a}.

\begin{lemma}
\label{t4.1.11}Let $\xi$ be an $(1,p)$-icon on $M$, and $\xi_{1}$ and $\xi
_{2}$ $(2,q)$-icons on $M$ with
\begin{align*}
&  \left(  \mathrm{id}_{\left[  M\otimes\mathcal{W}_{D^{q}}\rightarrow
M\right]  }\otimes\mathcal{W}_{\left(  d_{1},d_{2}\right)  \in D(2)\mapsto
\left(  d_{1},d_{2}\right)  \in D^{2}}\right)  \left(  \xi_{1}\right) \\
&  =\left(  \mathrm{id}_{\left[  M\otimes\mathcal{W}_{D^{q}}\rightarrow
M\right]  }\otimes\mathcal{W}_{\left(  d_{1},d_{2}\right)  \in D(2)\mapsto
\left(  d_{1},d_{2}\right)  \in D^{2}}\right)  \left(  \xi_{2}\right)
\end{align*}
Then the following formulas are both meaningful and valid.
\begin{align*}
\xi\circledast\xi_{1}\underset{1}{\overset{\cdot}{-}}\xi\circledast\xi_{2}  &
=\xi\circledast(\xi_{1}\overset{\cdot}{-}\xi_{2})\\
\xi\widetilde{\circledast}\xi_{1}\underset{1}{\overset{\cdot}{-}}\xi
\widetilde{\circledast}\xi_{2}  &  =\xi\widetilde{\circledast}(\xi_{1}%
\overset{\cdot}{-}\xi_{2})\\
\xi_{1}\circledast\xi\underset{3}{\overset{\cdot}{-}}\xi_{2}\circledast\xi &
=(\xi_{1}\overset{\cdot}{-}\xi_{2})\circledast\xi\\
\xi_{1}\widetilde{\circledast}\xi\underset{3}{\overset{\cdot}{-}}\xi
_{2}\widetilde{\circledast}\xi &  =(\xi_{1}\overset{\cdot}{-}\xi
_{2})\widetilde{\circledast}\xi
\end{align*}

\end{lemma}

\begin{proof}
(of Theorem \ref{t4.1.10}). Our present discussion is a tiny generalization of
Proposition 2.7 in \cite{nishi-a}. We define six $(3,p+q+r)$-icons on $M$ as
follows:
\begin{align*}
\xi_{123} &  =\xi_{1}\circledast\xi_{2}\circledast\xi_{3}\\
\xi_{132} &  =\xi_{1}\circledast(\xi_{2}\widetilde{\circledast}\xi_{3})\\
\xi_{213} &  =(\xi_{1}\widetilde{\circledast}\xi_{2})\circledast\xi_{3}\\
\xi_{231} &  =\xi_{1}\widetilde{\circledast}(\xi_{2}\circledast\xi_{3})\\
\xi_{312} &  =(\xi_{1}\circledast\xi_{2})\widetilde{\circledast}\xi_{3}\\
\xi_{321} &  =\xi_{1}\widetilde{\circledast}\xi_{2}\widetilde{\circledast}%
\xi_{3}%
\end{align*}
Then it is easy, by dint of Lemma \ref{t4.1.11}, to see that
\begin{align}
\left[  \xi_{1},\left[  \xi_{2},\xi_{3}\right]  _{L}\right]  _{L} &
=(\xi_{123}\underset{1}{\overset{\cdot}{-}}\xi_{132})\overset{\cdot}{-}%
(\xi_{231}\underset{1}{\overset{\cdot}{-}}\xi_{321})\label{J1}\\
\left(  \left[  \xi_{2},\left[  \xi_{3},\xi_{1}\right]  _{L}\right]
_{L}\right)  ^{\sigma_{p,q+r}} &  =(\xi_{231}\underset{2}{\overset{\cdot}{-}%
}\xi_{213})\overset{\cdot}{-}(\xi_{312}\underset{2}{\overset{\cdot}{-}}%
\xi_{132})\label{J2}\\
\left(  \left[  \xi_{3},\left[  \xi_{1},\xi_{2}\right]  _{L}\right]
_{L}\right)  ^{\sigma_{r,p+q}} &  =(\xi_{312}\underset{3}{\overset{\cdot}{-}%
}\xi_{321})\overset{\cdot}{-}(\xi_{123}\underset{3}{\overset{\cdot}{-}}%
\xi_{213})\label{J3}%
\end{align}
Therefore the desired Jacobi identity follows directly from the general Jacobi identity.
\end{proof}

\begin{remark}
In order to see that the right-hand side of (\ref{J1}) is meaningful, we have
to check that all of
\begin{align*}
&  \xi_{123}\underset{1}{\overset{\cdot}{-}}\xi_{132}\\
&  \xi_{231}\underset{1}{\overset{\cdot}{-}}\xi_{321}\\
&  (\xi_{123}\underset{1}{\overset{\cdot}{-}}\xi_{132})\overset{\cdot}{-}%
(\xi_{231}\underset{1}{\overset{\cdot}{-}}\xi_{321})
\end{align*}
are meaningful. Since $\xi_{2}\circledast\xi_{3}\overset{\cdot}{-}\xi
_{2}\widetilde{\circledast}\xi_{3}$ is meaningful by Lemma \ref{t4.1.5},
$\xi_{123}\underset{1}{\overset{\cdot}{-}}\xi_{132}$ is also meaningful and we
have
\[
\xi_{123}\underset{1}{\overset{\cdot}{-}}\xi_{132}=\xi_{1}\circledast(\xi
_{2}\circledast\xi_{3}\overset{\cdot}{-}\xi_{2}\widetilde{\circledast}\xi_{3})
\]
by Lemma \ref{t4.1.11}. Similarly $\xi_{231}\underset{1}{\overset{\cdot}{-}%
}\xi_{321}$ is meaningful and we have
\[
\xi_{231}\underset{1}{\overset{\cdot}{-}}\xi_{321}=\xi_{1}\widetilde
{\circledast}(\xi_{2}\circledast\xi_{3}\overset{\cdot}{-}\xi_{2}%
\widetilde{\circledast}\xi_{3})
\]
Therefore $(\xi_{123}\underset{1}{\overset{\cdot}{-}}\xi_{132})\overset{\cdot
}{-}(\xi_{231}\underset{1}{\overset{\cdot}{-}}\xi_{321})$ is meaningful by
Lemma \ref{t4.1.5}. Similar considerations apply to (\ref{J2}) and (\ref{J3}).
\end{remark}

\subsection{The Jacobi Identity for the Fr\"{o}licher-Nijenhuis Bracket}

\begin{definition}
Given a $(1,p)$-icon $\xi$\ on $M$, we define another $(1,p)$-icon
$\mathcal{A}\xi$\ on $M$ to be
\[
\mathcal{A}\xi=\sum_{\sigma\in\mathbb{S}_{p}}\varepsilon_{\sigma}\xi^{\sigma}%
\]

\end{definition}

\begin{notation}
Given a $(1,p+q)$-icon $\xi$\ on $M$, we write $\mathcal{A}_{p,q}\xi$ for
$(1/p!q!)\mathcal{A}\xi$. Given a $(1,p+q+r)$-icon $\xi$\ on $M$, We write
$\mathcal{A}_{p,q,r}\xi$ for $(1/p!q!r!)\mathcal{A}\xi$.
\end{notation}

\begin{lemma}
\label{t4.2.1}If $\xi_{1}$ is a tangent-vector-valued $p$-form on $M$,
$\xi_{2}$ is a tangent-vector-valued $q$-form on $M$ and $\xi_{3}$ is a
tangent-vector-valued $r$-form on $M$, then we have
\[
\mathcal{A}_{p,q+r}(\left[  \xi_{1},\mathcal{A}_{q,r}\left(  \left[  \xi
_{2},\xi_{3}\right]  _{L}\right)  \right]  _{L})=\mathcal{A}_{p,q,r}(\left[
\xi_{1},\left[  \xi_{2},\xi_{3}\right]  _{L}\right]  _{L})
\]

\end{lemma}

\begin{proof}
By the same token as in establishing the familiar associativity of wedge
products in differential forms.
\end{proof}

\begin{definition}
Given a tangent-vector-valued $p$\textit{-form }$\xi_{1}$ on $M$ and a
tangent-vector-valued $q$-form $\xi_{2}$ on $M$, we are going to define their
\textit{Fr\"{o}licher-Nijenhuis bracket} $\left[  \xi_{1},\xi_{2}\right]
_{FN}$ to be
\[
\left[  \xi_{1},\xi_{2}\right]  _{FN}=\mathcal{A}_{p,q}\mathcal{(}\left[
\xi_{1},\xi_{2}\right]  _{L})
\]
which is undoubtedly a tangent-vector-valued $(p+q)$-form on $M$.
\end{definition}

\begin{proposition}
\label{t4.2.2}If $\xi_{1}$ is a tangent-vector-valued $p$-form on $M$ and
$\xi_{2}$ is a tangent-vector-valued $q$-form on $M$, then we have the
following graded antisymmetry:
\[
\left[  \xi_{1},\xi_{2}\right]  _{FN}=-(-1)^{pq}\left[  \xi_{2},\xi
_{1}\right]  _{FN}%
\]

\end{proposition}

\begin{proof}
We have
\begin{align*}
&  \left[  \xi_{1},\xi_{2}\right]  _{FN}\\
&  =\mathcal{A}_{p,q}\mathcal{(}\left[  \xi_{1},\xi_{2}\right]  _{L})\\
&  =-\mathcal{A}_{p,q}\mathcal{(}\left(  \left[  \xi_{2},\xi_{1}\right]
_{L}\right)  ^{\sigma_{p,q}})\text{ \ \ [By Proposition \ref{t4.1.9}]}\\
&  =-\frac{1}{p!q!}\sum_{\tau\in\mathbb{S}_{p+q}}\varepsilon_{\tau}\left(
\left(  \left[  \xi_{2},\xi_{1}\right]  _{L}\right)  ^{\sigma_{p,q}}\right)
^{\tau}\\
&  =-\frac{1}{p!q!}\sum_{\tau\in\mathbb{S}_{p+q}}\varepsilon_{\tau}\left(
\left[  \xi_{2},\xi_{1}\right]  _{L}\right)  ^{\tau\sigma_{p,q}}\\
&  =-\frac{1}{p!q!}\varepsilon_{\sigma_{p,q}}\sum_{\tau\in\mathbb{S}_{p+q}%
}\varepsilon_{\tau\sigma_{p,q}}\left(  \left[  \xi_{2},\xi_{1}\right]
_{L}\right)  ^{\tau\sigma_{p,q}}\\
&  =-\varepsilon_{\sigma_{p,q}}\left[  \xi_{2},\xi_{1}\right]  _{FN}%
\end{align*}
Since $\varepsilon_{\rho}=(-1)^{pq}$, the desired conclusion follows.
\end{proof}

\begin{theorem}
\label{t4.2.3}If $\xi_{1}$ is a tangent-vector-valued $p$-form on $M$,
$\xi_{2}$ is a tangent-vector-valued $q$-form on $M$ and $\xi_{3}$ is a
tangent-vector-valued $r$-form on $M$, then the following graded Jacobi
identity holds:
\[
\left[  \xi_{1},\left[  \xi_{2},\xi_{3}\right]  _{FN}\right]  _{FN}%
+(-1)^{p(q+r)}\left[  \xi_{2},\left[  \xi_{3},\xi_{1}\right]  _{FN}\right]
_{FN}+(-1)^{r(p+q)}\left[  \xi_{3},\left[  \xi_{1},\xi_{2}\right]
_{FN}\right]  _{FN}=0
\]

\end{theorem}

\begin{proof}
We have
\begin{align*}
&  \left[  \xi_{1},\left[  \xi_{2},\xi_{3}\right]  _{FL}\right]
_{FL}+(-1)^{p(q+r)}\left[  \xi_{2},\left[  \xi_{3},\xi_{1}\right]
_{FL}\right]  _{FL}+(-1)^{r(p+q)}\left[  \xi_{3},\left[  \xi_{1},\xi
_{2}\right]  _{FL}\right]  _{FL}\\
&  =\mathcal{A}_{p,q+r}(\left[  \xi_{1},\mathcal{A}_{q,r}\left(  \left[
\xi_{2},\xi_{3}\right]  _{L}\right)  \right]  _{L})+(-1)^{p(q+r)}%
\mathcal{A}_{q,p+r}(\left[  \xi_{2},\mathcal{A}_{p,r}\left(  \left[  \xi
_{3},\xi_{1}\right]  _{L}\right)  \right]  _{L})+\\
&  (-1)^{r(p+q)}\mathcal{A}_{r,p+q}(\left[  \xi_{3},\mathcal{A}_{p,q}\left(
\left[  \xi_{1},\xi_{2}\right]  _{L}\right)  \right]  _{L})\\
&  =\mathcal{A}_{p,q,r}\left\{  \left[  \xi_{1},\left[  \xi_{2},\xi
_{3}\right]  _{L}\right]  _{L}+(-1)^{p(q+r)}\left[  \xi_{2},\left[  \xi
_{3},\xi_{1}\right]  _{L}\right]  _{L}+(-1)^{r(p+q)}\left[  \xi_{3},\left[
\xi_{1},\xi_{2}\right]  _{L}\right]  _{L}\right\} \\
&  \text{[By Lemma \ref{t4.2.1}]}\\
&  =\mathcal{A}_{p,q,r}\left\{  \left[  \xi_{1},\left[  \xi_{2},\xi
_{3}\right]  _{L}\right]  _{L}+\left(  \left[  \xi_{2},\left[  \xi_{3},\xi
_{1}\right]  _{L}\right]  _{L}\right)  ^{\sigma_{p,q+r}}+\left(  \left[
\xi_{3},\left[  \xi_{1},\xi_{2}\right]  _{L}\right]  _{L}\right)
^{\sigma_{r,p+q}}\right\} \\
&  =0\\
&  \text{\lbrack By Theorem \ref{t4.1.10}]}%
\end{align*}

\end{proof}

\section{The Lie Derivation}

\subsection{The Lie Derivation of the First Type}

\begin{definition}
Given $\eta\in\left[  M\otimes\mathcal{W}_{D^{p}}\rightarrow M\right]  $ and
$\theta\in\left[  M\otimes\mathcal{W}_{D^{q}}\rightarrow\mathbb{R}\right]  $,
their convolution $\eta\widetilde{\ast}\theta\in\left[  M\otimes
\mathcal{W}_{D^{p+q}}\rightarrow\mathbb{R}\right]  $ is defined to be the
outcome of the composition of mappings
\begin{align*}
M\otimes\mathcal{W}_{D^{p+q}}  &  =M\otimes\left(  \mathcal{W}_{D^{p}}%
\otimes_{\infty}\mathcal{W}_{D^{q}}\right) \\
&  =\left(  M\otimes\mathcal{W}_{D^{p}}\right)  \otimes\mathcal{W}_{D^{q}}
\begin{array}
[c]{c}%
\eta\otimes\mathrm{id}_{\mathcal{W}_{D^{q}}}\\
\rightarrow\\
\,
\end{array}
M\otimes\mathcal{W}_{D^{q}}
\begin{array}
[c]{c}%
\theta\\
\rightarrow\\
\,
\end{array}
\mathbb{R}%
\end{align*}

\end{definition}

It should be obvious that

\begin{proposition}
\label{t5.1.1}Given $\eta_{1}\in\left[  M\otimes\mathcal{W}_{D^{p}}\rightarrow
M\right]  $, $\eta_{2}\in\left[  M\otimes\mathcal{W}_{D^{q}}\rightarrow
M\right]  $ and $\theta\in\left[  M\otimes\mathcal{W}_{D^{r}}\rightarrow
\mathbb{R}\right]  $, we have
\[
(\eta_{1}\widetilde{\ast}\eta_{2})\widetilde{\ast}\theta=\eta_{1}%
\widetilde{\ast}(\eta_{2}\widetilde{\ast}\theta)
\]

\end{proposition}

\begin{definition}
We define a binary mapping
\begin{align*}
\widetilde{\circledast}  &  :\left(  \left[  M\otimes\mathcal{W}_{D^{p}%
}\rightarrow M\right]  \otimes\mathcal{W}_{D^{m}}\right)  \times\left(
\left[  M\otimes\mathcal{W}_{D^{q}}\rightarrow\mathbb{R}\right]
\otimes\mathcal{W}_{D^{n}}\right)  \rightarrow\\
&  \left[  M\otimes\mathcal{W}_{D^{p+q}}\rightarrow\mathbb{R}\right]
\otimes\mathcal{W}_{D^{m+n}}%
\end{align*}
to be the composition of mappings
\begin{align*}
&  \left(  \left[  M\otimes\mathcal{W}_{D^{p}}\rightarrow M\right]
\otimes\mathcal{W}_{D^{m}}\right)  \times\left(  \left[  M\otimes
\mathcal{W}_{D^{q}}\rightarrow\mathbb{R}\right]  \otimes\mathcal{W}_{D^{n}%
}\right) \\
&  \underline{\left(  \mathrm{id}_{\left[  M\otimes\mathcal{W}_{D^{p}%
}\rightarrow M\right]  }\otimes\mathcal{W}_{\left(  d_{1},...,d_{m}%
,d_{m+1},...,d_{m+n}\right)  \in D^{m+n}\mapsto\left(  d_{1},...,d_{m}\right)
\in D^{m}}\right)  \times}\\
&  \underrightarrow{\left(  \mathrm{id}_{\left[  M\otimes\mathcal{W}_{D^{q}%
}\rightarrow\mathbb{R}\right]  }\otimes\mathcal{W}_{\left(  d_{1}%
,...,d_{m},d_{m+1},...,d_{m+n}\right)  \in D^{m+n}\mapsto\left(
d_{m+1},...,d_{m+n}\right)  \in D^{n}}\right)  }\\
&  \left(  \left[  M\otimes\mathcal{W}_{D^{p}}\rightarrow M\right]
\otimes\mathcal{W}_{D^{m+n}}\right)  \times\left(  \left[  M\otimes
\mathcal{W}_{D^{q}}\rightarrow\mathbb{R}\right]  \otimes\mathcal{W}_{D^{m+n}%
}\right) \\
&  =\left(  \left[  M\otimes\mathcal{W}_{D^{p}}\rightarrow M\right]
\times\left[  M\otimes\mathcal{W}_{D^{q}}\rightarrow\mathbb{R}\right]
\right)  \otimes\mathcal{W}_{D^{m+n}}\\
&  \underrightarrow{\widetilde{\ast}\otimes\mathrm{id}_{\mathcal{W}_{D^{m+n}}%
}}\left[  M\otimes\mathcal{W}_{D^{p+q}}\rightarrow\mathbb{R}\right]
\otimes\mathcal{W}_{D^{m+n}}%
\end{align*}

\end{definition}

It should be obvious that

\begin{proposition}
\label{t5.1.2}Given $\xi_{1}\in\left[  M\otimes\mathcal{W}_{D^{p}}\rightarrow
M\right]  \otimes\mathcal{W}_{D^{m}}$, $\xi_{2}\in\left[  M\otimes
\mathcal{W}_{D^{q}}\rightarrow M\right]  \otimes\mathcal{W}_{D^{n}}$ and
$\theta\in\left[  M\otimes\mathcal{W}_{D^{r}}\rightarrow\mathbb{R}\right]
\otimes\mathcal{W}_{D^{l}}$, we have
\[
(\xi_{1}\widetilde{\circledast}\xi_{2})\widetilde{\circledast}\mathbb{\theta
}=\xi_{1}\widetilde{\circledast}(\xi_{2}\widetilde{\circledast}\theta)
\]

\end{proposition}

\begin{definition}
For any $\xi\in\left[  M\otimes\mathcal{W}_{D^{p}}\rightarrow M\right]
\otimes\mathcal{W}_{D}$ and any $\theta\in\left[  M\otimes\mathcal{W}_{D^{q}%
}\rightarrow\mathbb{R}\right]  $, we define $\widehat{\mathcal{L}}_{\xi}%
\theta$ to be
\[
\widehat{\mathcal{L}}_{\xi}\theta=\mathbf{D}\left(  \xi\widetilde{\otimes
}\theta\right)
\]

\end{definition}

It is easy to see that

\begin{proposition}
\label{t5.1.3}Given $\xi\in\left[  M\otimes\mathcal{W}_{D^{p}}\rightarrow
M\right]  \otimes\mathcal{W}_{D}$, $\theta_{1}\in\left[  M\otimes
\mathcal{W}_{D^{q}}\rightarrow\mathbb{R}\right]  $ and $\theta_{2}\in\left[
M\otimes\mathcal{W}_{D^{r}}\rightarrow\mathbb{R}\right]  $, we have
\[
\widehat{\mathcal{L}}_{\xi}\left(  \theta_{1}\otimes\theta_{2}\right)
=\left(  \widehat{\mathcal{L}}_{\xi}\theta_{1}\right)  \otimes\theta
_{2}+\left(  \theta_{1}\otimes\left(  \widehat{\mathcal{L}}_{\xi}\theta
_{2}\right)  \right)  ^{\sigma_{p,q}^{+r}}%
\]
where $\sigma_{p,q}^{+r}$\ is
\[
\left(
\begin{array}
[c]{ccccccccc}%
1 & ... & p & p+1 & ... & p+q & p+q+1 & ... & p+q+r\\
q+1 & ... & p+q & 1 & ... & q & p+q+1 & ... & p+q+r
\end{array}
\right)
\]

\end{proposition}

It should be obvious that

\begin{proposition}
\label{t5.1.4}If $\xi\in\left[  M\otimes\mathcal{W}_{D^{p}}\rightarrow
M\right]  \otimes\mathcal{W}_{D}$ is a tangent-vector-valued $p$-semiform and
$\theta\in\left[  M\otimes\mathcal{W}_{D^{q}}\rightarrow\mathbb{R}\right]  $
is a $q$-semiform, then $\widehat{\mathcal{L}}_{\xi}\theta$ is a $(p+q)$-semiform.
\end{proposition}

\begin{remark}
Therefore, given a tangent-vector-valued $p$-semiform $\xi$\ on $M$,
$\widehat{\mathcal{L}}_{\xi}$ is considered to be a graded mapping of degree
$p$\ on the space $\widetilde{\Omega}\left(  M\right)  $.
\end{remark}

\begin{proposition}
\label{t5.1.5}If $\xi,\xi_{1},\xi_{2}\in\left[  M\otimes\mathcal{W}_{D^{p}%
}\rightarrow M\right]  \otimes\mathcal{W}_{D}$ are tangent-vector-valued
$p$-semiforms, $\theta,\theta_{1},\theta_{2}\in\left[  M\otimes\mathcal{W}%
_{D^{q}}\rightarrow\mathbb{R}\right]  $ are $q$-semiforms and $\alpha
\in\mathbb{R}$, then we have the following:

\begin{enumerate}
\item
\[
\widehat{\mathcal{L}}_{\xi_{1}+\xi_{2}}\theta=\widehat{\mathcal{L}}_{\xi_{1}%
}\theta+\widehat{\mathcal{L}}_{\xi_{2}}\theta
\]

\item
\[
\widehat{\mathcal{L}}_{\alpha\xi}\theta=\alpha\left(  \widehat{\mathcal{L}%
}_{\xi}\theta\right)
\]

\item
\[
\widehat{\mathcal{L}}_{\xi}\left(  \theta_{1}+\theta_{2}\right)
=\widehat{\mathcal{L}}_{\xi}\theta_{1}+\widehat{\mathcal{L}}_{\xi}\theta_{2}%
\]

\item
\[
\widehat{\mathcal{L}}_{\xi}\left(  \alpha\theta\right)  =\alpha\left(
\widehat{\mathcal{L}}_{\xi}\theta\right)
\]

\end{enumerate}
\end{proposition}

\begin{proof}
The statements 2 and 4 follow from the definitions. The statement 1 follows
from the statement 2, while the statement 3 follows from the statement 4.
\end{proof}

\begin{theorem}
\label{t5.1.6}If $\xi_{1}\in\left[  M\otimes\mathcal{W}_{D^{p}}\rightarrow
M\right]  \otimes\mathcal{W}_{D}$ is a tangent-vector-valued $p$-semiform and
$\xi_{2}\in\left[  M\otimes\mathcal{W}_{D^{q}}\rightarrow M\right]
\otimes\mathcal{W}_{D}$ is a tangent-vector-valued $q$-semiform, then we have
\[
\widehat{\mathcal{L}}_{\left[  \xi_{1},\xi_{2}\right]  _{L}}=\left[
\widehat{\mathcal{L}}_{\xi_{1}},\widehat{\mathcal{L}}_{\xi_{2}}\right]
=\widehat{\mathcal{L}}_{\xi_{1}}\circ\widehat{\mathcal{L}}_{\xi_{2}}%
-(-1)^{pq}\widehat{\mathcal{L}}_{\xi_{2}}\circ\widehat{\mathcal{L}}_{\xi_{1}}%
\]

\end{theorem}

\begin{proof}
If $\theta\in\left[  M\otimes\mathcal{W}_{D^{r}}\rightarrow\mathbb{R}\right]
$ is a $r$-semiform, then we have
\begin{align*}
&  \widehat{\mathcal{L}}_{\left[  \xi_{1},\xi_{2}\right]  _{L}}\theta\\
&  =\mathbf{D}\left(  \left[  \xi_{1},\xi_{2}\right]  _{L}\widetilde{\otimes
}\mathbb{\theta}\right) \\
&  =\mathbf{D}\left(  \left(  \xi_{1}\widetilde{\otimes}\xi_{2}\overset{\cdot
}{-}\xi_{1}\otimes\xi_{2}\right)  \widetilde{\circledast}\mathbb{\theta
}\right) \\
&  =\mathbf{D}\left(  \left(  \xi_{1}\widetilde{\otimes}\xi_{2}\right)
\widetilde{\otimes}\mathbb{\theta}\overset{\cdot}{-}\left(  \xi_{1}\otimes
\xi_{2}\right)  \widetilde{\otimes}\mathbb{\theta}\right) \\
&  =\mathbf{D}\left(  \left(  \xi_{1}\widetilde{\otimes}\xi_{2}\right)
\widetilde{\otimes}\mathbb{\theta}\overset{\cdot}{-}\left(  \xi_{2}%
\widetilde{\otimes}\xi_{1}\right)  _{\tau}^{\sigma_{p,q}}\widetilde{\otimes
}\mathbb{\theta}\right) \\
&  =\mathbf{D}\left(
\begin{array}
[c]{c}%
\xi_{1}\widetilde{\otimes}\left(  \xi_{2}\widetilde{\otimes}\mathbb{\theta
}\right)  \overset{\cdot}{-}\\
\left(  \mathrm{id}_{\widetilde{\Omega}^{p+q+r}\left(  M\right)  }%
\otimes\mathcal{W}_{\left(  d_{1},d_{2}\right)  \in D^{2}\mapsto\left(
d_{2},d_{1}\right)  \in D^{2}}\right)  \left(  \left(  \xi_{2}\widetilde
{\otimes}\left(  \xi_{1}\widetilde{\otimes}\mathbb{\theta}\right)  \right)
^{\sigma_{p,q}^{+r}}\right)
\end{array}
\right) \\
&  =\mathbf{D}\left(  \mathbf{D}_{2}\left(  \xi_{1}\widetilde{\otimes}\left(
\xi_{2}\widetilde{\otimes}\mathbb{\theta}\right)  \right)  \right)
-(-1)^{pq}\mathbf{D}\left(  \mathbf{D}_{1}\left(  \xi_{2}\widetilde{\otimes
}\left(  \xi_{1}\widetilde{\otimes}\mathbb{\theta}\right)  \right)  \right) \\
&  \text{[By Proposition \ref{t2.4.3}]}\\
&  =\mathbf{D}\left(  \xi_{1}\widetilde{\otimes}\mathbf{D}\left(  \xi
_{2}\widetilde{\otimes}\mathbb{\theta}\right)  \right)  -(-1)^{pq}%
\mathbf{D}\left(  \xi_{2}\widetilde{\otimes}\mathbf{D}\left(  \xi
_{1}\widetilde{\otimes}\mathbb{\theta}\right)  \right) \\
&  =\widehat{\mathcal{L}}_{\xi_{1}}\left(  \widehat{\mathcal{L}}_{\xi_{2}%
}\theta\right)  -(-1)^{pq}\widehat{\mathcal{L}}_{\xi_{2}}\left(
\widehat{\mathcal{L}}_{\xi_{1}}\theta\right)
\end{align*}

\end{proof}

\subsection{The Lie Derivation of the Second Type}

\begin{definition}
For any $\xi\in\left[  M\otimes\mathcal{W}_{D^{p}}\rightarrow M\right]
\otimes\mathcal{W}_{D}$ and any $\theta\in\left[  M\otimes\mathcal{W}_{D^{q}%
}\rightarrow\mathbb{R}\right]  $, we define $\mathcal{L}_{\xi}\theta$ to be
\[
\mathcal{L}_{\xi}\theta=\mathcal{A}_{p,q}\left(  \widehat{\mathcal{L}}_{\xi
}\theta\right)
\]

\end{definition}

It should be obvious that

\begin{proposition}
\label{t5.2.1}If $\xi\in\left[  M\otimes\mathcal{W}_{D^{p}}\rightarrow
M\right]  \otimes\mathcal{W}_{D}$ is a tangent-vector-valued $p$-form and
$\theta\in\left[  M\otimes\mathcal{W}_{D^{q}}\rightarrow\mathbb{R}\right]  $
is a $q$-form, then $\mathcal{L}_{\xi}\theta$ is a $(p+q)$-form.
\end{proposition}

\begin{proposition}
\label{t5.2.2}Given $\xi\in\left[  M\otimes\mathcal{W}_{D^{p}}\rightarrow
M\right]  \otimes\mathcal{W}_{D}$, $\theta_{1}\in\left[  M\otimes
\mathcal{W}_{D^{q}}\rightarrow\mathbb{R}\right]  $ and $\theta_{2}\in\left[
M\otimes\mathcal{W}_{D^{r}}\rightarrow\mathbb{R}\right]  $, we have the following:

\begin{enumerate}
\item
\[
\mathcal{A}_{p,q,r}\left(  \left(  \widehat{\mathcal{L}}_{\xi}\theta
_{1}\right)  \otimes\theta_{2}\right)  =\mathcal{A}_{p+q,r}\left(
\mathcal{A}_{p,q}\left(  \widehat{\mathcal{L}}_{\xi}\theta_{1}\right)
\otimes\theta_{2}\right)
\]

\item
\[
\mathcal{A}_{p,q,r}\left(  \theta_{1}\otimes\left(  \widehat{\mathcal{L}}%
_{\xi}\theta_{2}\right)  \right)  =\mathcal{A}_{q,p+r}\left(  \theta
_{1}\otimes\mathcal{A}_{p,r}\left(  \widehat{\mathcal{L}}_{\xi}\theta
_{2}\right)  \right)
\]

\end{enumerate}
\end{proposition}

\begin{proof}
By the same token as that in establishing the familiar associativity of wedge
products in differential forms.
\end{proof}

\begin{proposition}
\label{t5.2.3}Given $\xi\in\left[  M\otimes\mathcal{W}_{D^{p}}\rightarrow
M\right]  \otimes\mathcal{W}_{D}$, $\theta_{1}\in\left[  M\otimes
\mathcal{W}_{D^{q}}\rightarrow\mathbb{R}\right]  $ and $\theta_{2}\in\left[
M\otimes\mathcal{W}_{D^{r}}\rightarrow\mathbb{R}\right]  $, we have
\[
\mathcal{L}_{\xi}\left(  \theta_{1}\wedge\theta_{2}\right)  =\left(
\mathcal{L}_{\xi}\theta_{1}\right)  \wedge\theta_{2}+(-1)^{pq}\theta_{1}%
\wedge\left(  \mathcal{L}_{\xi}\theta_{2}\right)
\]

\end{proposition}

\begin{proof}
We proceed as follows:
\begin{align*}
&  \mathcal{L}_{\xi}\left(  \theta_{1}\wedge\theta_{2}\right) \\
&  =\mathcal{A}_{p,q+r}\left(  \widehat{\mathcal{L}}_{\xi}\left(
\mathcal{A}_{q,r}\left(  \theta_{1}\otimes\theta_{2}\right)  \right)  \right)
\\
&  =\mathcal{A}_{p,q,r}\left(  \widehat{\mathcal{L}}_{\xi}\left(  \theta
_{1}\otimes\theta_{2}\right)  \right) \\
&  =\mathcal{A}_{p,q,r}\left(  \left(  \widehat{\mathcal{L}}_{\xi}\theta
_{1}\right)  \otimes\theta_{2}+\left(  \theta_{1}\otimes\left(  \widehat
{\mathcal{L}}_{\xi}\theta_{2}\right)  \right)  ^{\sigma_{p,q}^{+r}}\right) \\
&  \text{[By Proposition \ref{t5.1.3}]}\\
&  =\mathcal{A}_{p,q,r}\left(  \left(  \widehat{\mathcal{L}}_{\xi}\theta
_{1}\right)  \otimes\theta_{2}\right)  +\mathcal{A}_{p,q,r}\left(  \left(
\theta_{1}\otimes\left(  \widehat{\mathcal{L}}_{\xi}\theta_{2}\right)
\right)  ^{\sigma_{p,q}^{+r}}\right) \\
&  =\mathcal{A}_{p+q,r}\left(  \mathcal{A}_{p,q}\left(  \widehat{\mathcal{L}%
}_{\xi}\theta_{1}\right)  \otimes\theta_{2}\right)  +(-1)^{pq}\mathcal{A}%
_{q,p+r}\left(  \theta_{1}\otimes\mathcal{A}_{p,r}\left(  \widehat
{\mathcal{L}}_{\xi}\theta_{2}\right)  \right) \\
&  \text{[By Proposition \ref{t5.2.2}]}\\
&  =\left(  \mathcal{L}_{\xi}\theta_{1}\right)  \wedge\theta_{2}%
+(-1)^{pq}\theta_{1}\wedge\left(  \mathcal{L}_{\xi}\theta_{2}\right)
\end{align*}

\end{proof}

\begin{remark}
Therefore, given a tangent-vector-valued $p$-form $\xi$\ on $M$,
$\mathcal{L}_{\xi}$ is considered to be a graded mapping of degree $p$\ on the
space $\Omega\left(  M\right)  $.
\end{remark}

\begin{proposition}
\label{t5.2.4}For any $\xi\in\left[  M\otimes\mathcal{W}_{D^{p+q}}\rightarrow
M\right]  \otimes\mathcal{W}_{D}$, any $\xi_{1}\in\left[  M\otimes
\mathcal{W}_{D^{p}}\rightarrow M\right]  \otimes\mathcal{W}_{D}$, any $\xi
_{2}\in\left[  M\otimes\mathcal{W}_{D^{q}}\rightarrow M\right]  \otimes
\mathcal{W}_{D}$, and any $\theta\in\left[  M\otimes\mathcal{W}_{D^{r}%
}\rightarrow\mathbb{R}\right]  $, we have the following:

\begin{enumerate}
\item
\[
\mathcal{A}_{p,q,r}\left(  \widehat{\mathcal{L}}_{\xi}\theta\right)
=\mathcal{A}_{p+q,r}\left(  \widehat{\mathcal{L}}_{\mathcal{A}_{p,q}\left(
\xi\right)  }\theta\right)
\]

\item
\[
\mathcal{A}_{p,q,r}\left(  \widehat{\mathcal{L}}_{\xi_{1}}\left(
\widehat{\mathcal{L}}_{\xi_{2}}\theta\right)  \right)  =\mathcal{A}%
_{p,q+r}\left(  \widehat{\mathcal{L}}_{\xi_{1}}\mathcal{A}_{q,r}\left(
\widehat{\mathcal{L}}_{\xi_{2}}\theta\right)  \right)
\]

\item
\[
\mathcal{A}_{p,q,r}\left(  \widehat{\mathcal{L}}_{\xi_{2}}\left(
\widehat{\mathcal{L}}_{\xi_{1}}\theta\right)  \right)  =\mathcal{A}%
_{q,p+r}\left(  \widehat{\mathcal{L}}_{\xi_{2}}\mathcal{A}_{p,r}\left(
\widehat{\mathcal{L}}_{\xi_{1}}\theta\right)  \right)
\]

\end{enumerate}
\end{proposition}

\begin{proof}
By the same token as that in the familiar associativity of wedge product in
differential forms.
\end{proof}

\begin{theorem}
\label{t5.2.5}If both $\xi_{1}\in\left[  M\otimes\mathcal{W}_{D^{p}%
}\rightarrow M\right]  \otimes\mathcal{W}_{D}$ and $\xi_{2}\in\left[
M\otimes\mathcal{W}_{D^{q}}\rightarrow M\right]  \otimes\mathcal{W}_{D}$ are
tangent-vector-valued semiforms, then we have
\[
\mathcal{L}_{\left[  \xi_{1},\xi_{2}\right]  _{FN}}=\left[  \mathcal{L}%
_{\xi_{1}},\mathcal{L}_{\xi_{2}}\right]  =\mathcal{L}_{\xi_{1}}\circ
\mathcal{L}_{\xi_{2}}-(-1)^{pq}\mathcal{L}_{\xi_{2}}\circ\mathcal{L}_{\xi_{1}}%
\]

\end{theorem}

\begin{proof}
For any $\theta\in\left[  M\otimes\mathcal{W}_{D^{r}}\rightarrow
\mathbb{R}\right]  \otimes\mathcal{W}_{D}$, we have
\begin{align*}
&  \mathcal{L}_{\left[  \xi_{1},\xi_{2}\right]  _{FN}}\theta\\
&  =\mathcal{A}_{p+q,r}\left(  \widehat{\mathcal{L}}_{\mathcal{A}_{p,q}\left(
\left[  \xi_{1},\xi_{2}\right]  _{L}\right)  }\theta\right) \\
&  =\mathcal{A}_{p,q,r}\left(  \widehat{\mathcal{L}}_{\left[  \xi_{1},\xi
_{2}\right]  _{L}}\theta\right)  \text{ \ }\\
&  \text{[By the first statement of Proposition \ref{t5.2.4}]}\\
&  =\mathcal{A}_{p,q,r}\left(  \widehat{\mathcal{L}}_{\xi_{1}}\left(
\widehat{\mathcal{L}}_{\xi_{2}}\theta\right)  -(-1)^{pq}\widehat{\mathcal{L}%
}_{\xi_{2}}\left(  \widehat{\mathcal{L}}_{\xi_{1}}\theta\right)  \right) \\
&  \text{[By Theorem \ref{t5.1.5}]}\\
&  =\mathcal{A}_{p,q,r}\left(  \widehat{\mathcal{L}}_{\xi_{1}}\left(
\widehat{\mathcal{L}}_{\xi_{2}}\theta\right)  \right)  -(-1)^{pq}%
\mathcal{A}_{p,q,r}\left(  \widehat{\mathcal{L}}_{\xi_{2}}\left(
\widehat{\mathcal{L}}_{\xi_{1}}\theta\right)  \right) \\
&  =\mathcal{A}_{p,q+r}\left(  \widehat{\mathcal{L}}_{\xi_{1}}\mathcal{A}%
_{q,r}\left(  \widehat{\mathcal{L}}_{\xi_{2}}\theta\right)  \right)
-(-1)^{pq}\mathcal{A}_{q,p+r}\left(  \widehat{\mathcal{L}}_{\xi_{2}%
}\mathcal{A}_{p,r}\left(  \widehat{\mathcal{L}}_{\xi_{1}}\theta\right)
\right) \\
&  \text{[By the second and third statements of Proposition \ref{t5.2.4}]}\\
&  =\mathcal{L}_{\xi_{1}}\left(  \mathcal{L}_{\xi_{2}}\theta\right)
-(-1)^{pq}\mathcal{L}_{\xi_{2}}\left(  \mathcal{L}_{\xi_{1}}\theta\right)
\end{align*}

\end{proof}

\end{document}